\documentclass{amsart}
\usepackage[utf8]{inputenc}
\usepackage{upgreek,amsthm,bbm,bm}
\usepackage{color,amssymb, comment, mathrsfs,tikz,amsmath,amsfonts,bm,stmaryrd,tikz-cd,hyperref,appendix}
\usetikzlibrary{matrix}
\usetikzlibrary{arrows}
\usepackage[all]{xy}

\hypersetup{
    colorlinks=true,
    linkcolor=blue,
    filecolor=magenta,      
    urlcolor=cyan,
} 

\usepackage[capitalise]{cleveref}
\crefformat{equation}{(#2#1#3)}
\crefrangeformat{equation}{(#3#1#4--#5#2#6)}
\crefformat{enumi}{(#2#1#3)}
\crefrangeformat{enumi}{(#3#1#4--#5#2#6)}

\title[Support varieties over skew complete intersections]{Support varieties over skew complete intersections via derived braided Hochschild cohomology}

\author[L.~Ferraro]{Luigi Ferraro}
\address{Department of Mathematics,
Texas Tech University, Lubbock, TX 79409, U.S.A.}
\email{lferraro@ttu.edu}

\author[W.~F.~Moore]{W.~Frank Moore}
\address{Department of Mathematics \& Statistics,
Wake Forest University, Winstom-Salem, NC 27109, U.S.A.}
\email{moorewf@wfu.edu}

\author[J.~Pollitz]{Josh Pollitz}
\address{Department of Mathematics,
University of Utah, Salt Lake City, UT 84112, U.S.A.}
\email{pollitz@math.utah.edu}

\thanks{The third author was supported by the National Science Foundation under Grant No. 2002173.
}

\newtheorem{theorem}{Theorem}[section]
\newtheorem{lemma}[theorem]{Lemma}
\newtheorem{proposition}[theorem]{Proposition}
\newtheorem{corollary}[theorem]{Corollary}
\theoremstyle{definition}
\newtheorem{definition}[theorem]{Definition}
\newtheorem{construction}[theorem]{Construction}
\newtheorem{example}[theorem]{Example}

\newtheorem{notation}[theorem]{Notation}

\theoremstyle{remark}
\newtheorem{remark}[theorem]{Remark}

\newtheorem{chunk}[theorem]{}

\newcommand{\del}{\partial}
\DeclareMathOperator{\id}{id}
\DeclareMathOperator{\V}{V}
\DeclareMathOperator{\coker}{coker}
\newcommand{\pd}{\operatorname{pd}}
\newcommand{\op}{\operatorname{op}}

\newcommand{\xra}[1]{\xrightarrow{#1}}
\newcommand{\CC}{\mathbb{C}}
\newcommand{\f}{\bm{f}}

\newcommand{\Mod}{\mathsf{Mod}}

\newcommand{\D}{\mathsf{D}}
\newcommand{\h}{\mathsf{h}}
\newcommand{\te}{\mathsf{t}}
\renewcommand{\c}{\kappa}
\newcommand{\Ext}{\operatorname{Ext}}
\newcommand{\Tor}{\operatorname{Tor}}
\newcommand{\Der}{\operatorname{Der}}
\newcommand{\E}{\mathcal{E}}
\newcommand{\HH}{\mathrm{HH}}
\newcommand{\hh}{\mathrm{H}}
\DeclareMathOperator{\Proj}{Proj}
\DeclareMathOperator{\ann}{ann}
\DeclareMathOperator{\supp}{\mathsf{supp}}
\newcommand{\cV}{\mathcal{V}}
\newcommand{\cx}{\operatorname{cx}}

\newcommand{\kk}{\Bbbk}

\newcommand{\bsx}{\ensuremath{\mathbf{x}}}
\newcommand{\bsalpha}{{\ensuremath{\boldsymbol{\alpha}}}}
\newcommand{\bsbeta}{{\ensuremath{\boldsymbol{\beta}}}}
\newcommand{\bsgamma}{{\ensuremath{\boldsymbol{\gamma}}}}
\newcommand{\bsxi}{{\ensuremath{\boldsymbol{\xi}}}}
\newcommand{\combBracket}[2]{{\ensuremath{\begin{bmatrix}{#1}\\{#2}\end{bmatrix}}}}
\newcommand{\bbN}{\mathbb{N}}
\newcommand{\calG}{\mathscr{G}}

\newcommand{\e}{\mathrm{e}}
\newcommand{\ev}{\mathsf{ev}}
\newcommand{\Hom}{\operatorname{Hom}}

\newcommand{\dt}{\otimes^{\mathbf{L}}}
\newcommand{\dHH}[2]{\operatorname{\mathbf{HH}}(#1\!\mid\!#2)}  
\newcommand{\Ay}[2]{{#1}\!\left\langle{#2} \right\rangle}

\newcommand{\gdeg}[1]{\mathscr{G}(#1)}

\begin{document}

\begin{abstract}
In this article we study a theory of support varieties over a  skew complete intersection $R$, i.e. a skew polynomial ring modulo an ideal generated by a sequence of regular normal elements. We compute the derived braided Hochschild cohomology of $R$ relative to the skew polynomial ring and show its action  on $\Ext_R(M,N)$ is noetherian for finitely generated $R$-modules $M$ and $N$ respecting the braiding of $R$. When the parameters defining the skew polynomial ring are roots of unity we use this action to define a support theory. 
In this setting applications include a proof of the Generalized Auslander-Reiten Conjecture and that $R$ possesses  symmetric complexity.
\end{abstract}

\maketitle
\section*{Introduction}

The use of the cohomological spectrum has had a tremendous impact on modular representation theory; two notable works are those of 
 Carlson  \cite{carlson} and Quillen \cite{QuillenAnn}.
Inspired by these successes, similar theories of support have been developed and successfully applied for different kinds of algebras. For example, for restricted Lie algebras   \cite{FriedlanderParshall},  finite-dimensional graded connected cocommutative Hopf algebras  \cite{Hopf}, commutative   complete intersections \cite{AB}, and quantum complete intersections \cite{BerghErdmann}.


In this paper we define and study a support variety theory for skew complete intersections, a class of rings first studied in \cite{ColorDGA} by the first and second author; skew complete intersections contain the class of commutative graded complete intersections and  the class of quantum complete intersections. Let $\kk$ be a field and let $Q=\kk_\mathfrak{q}[x_1,\ldots,x_n]$ be a skew polynomial ring, i.e., a polynomial ring such that the variables skew commute $x_ix_j=q_{i,j}x_jx_i$, with $q_{i,j}\in\kk^*$, $q_{i,i}=1$  and $q_{i,j}=q_{j,i}^{-1}$ for all $i,j$. A skew complete intersection $R$ is quotient of $Q$ by  an ideal generated by a regular sequence of normal elements. To define these support varieties we use color DG homological algebra, a theory that was developed in \cite{ColorDGA}. The support  theory itself is inspired by ideas present in \cite{AB:CD2} and further developed in  \cite{P:CS}, to  transfer the  well behaved homological properties  over $Q$ to homological properties over $R$. In particular, we prescribe the graded Ext modules over $R$ with a ring of cohomology operators for which the action is noetherian.

We arrive at this ring of cohomological operators  in two ways.  First, we realize these operators as the  derived braided Hochschild cohomology of $R$ over $Q$. It is worth contrasting this with the approach of Bergh and Erdmann in \cite{BerghErdmann} who use  classical Hochschild cohomology to define support varieties over quantum complete intersections;  classical Hochschild cohomology can be  difficult to calculate and typically possesses a complicated structure. This is  even the case for  quantum complete intersections (cf. \cite{QuantumBE,BGMS,ErdmannFinnsen,Oppermann}). The  calculation of derived braided Hochschild cohomology of $R$ over $Q$, found in \Cref{s:HH}, leads to an easier calculation of a ring of cohomology operators that has a particularly simple structure, see \Cref{s:HH}. 

Second, these operators are realized at the chain level by following the strategies in \cite{AB:CD2,P:CS}. This is the perspective that allows us to define a support theory when each $q_{i,j}$ is a root of unity. Moreover, this approach is  in line with the support theory for commutative complete intersections (cf. \cite{VPD,AB,AI,BW,Jor,P:CS}), and generalizes the support variety theory for graded complete intersections.


Our main applications are  in \Cref{s:ARC} and \Cref{s:cx}. First, as the Ext modules over $R$ are noetherian over this ring of cohomology operators we generalize  results from  \cite{ColorDGA} which pertain to the Poincar\'{e} series of color $R$-modules. Further assuming  each of the $q_{i,j}$ defining $R$ is a root of unity, we can apply our theory of supports and obtain other interesting results. For example,  in \Cref{s:ARC}, we prove that color modules over such skew complete intersections satisfy the Generalized Auslander-Reiten Conjecture. Also, in \Cref{s:cx}, we  prove complexity for pairs of finitely generated color $R$-modules is symmetric in the module arguments for such skew complete intersections; therefore we deduce that if $M$ and $N$ are finitely generated color modules over a skew complete intersection $R$, then $\Ext_R^{\gg0}(M,N)=0$ if and only if $\Ext_R^{\gg0}(N,M)=0$. The latter result should be compared with the analogous result from 
 \cite{Bergh} for  modules over quantum complete intersections, with the same assumption on the $q_{i,j}$'s.

 The paper is organized as follows. In  \Cref{s:background} we recall background information on color commutative rings and DG algebras that can be found in \cite{ColorDGA}, and we provide the definition of the derived braided Hochschild cohomology of a ring. In \Cref{s:semifreediagonal} we recall the definition of skew complete intersection and make the necessary constructions needed to compute the derived braided Hochschild cohomology of a skew complete intersection in \Cref{s:HH}. In \Cref{section:severalactions} we show that two natural products on the derived braided Hochschild cohomology coincide. \Cref{s:extmodules} and \Cref{s:FiniteGen} are dedicated to defining our ring of cohomological operators at the chain level and  showing that Ext modules are finitely generated modules over this ring.  In \Cref{section:supportvarieties} we assume that the parameters $q_{i,j}$ are roots of unity, this allows us to work with a subring of cohomological operators that is   a commutative polynomial ring. We use this smaller ring to define support varieties.  The rest of \Cref{section:supportvarieties} is dedicated to the study of the properties of these support varieties and to several consequences. As discussed above, the main applications are in \Cref{s:ARC} and \Cref{s:cx}.

\section{Background and conventions}
\label{s:background}
\subsection*{Color commutative rings}
Let $G$ be an abelian group and $\kk$ a field. An \emph{alternating bicharacter} on $G$ is a function $\c:G\times G\rightarrow\kk^*$ such that for all $\alpha,\beta,\sigma\in G$
\begin{align*}
\c(\sigma,\alpha\beta)&=\c(\sigma,\alpha)\c(\sigma,\beta),\\
\c(\sigma,\sigma)&=1.
\end{align*}
Let $A$ be a $G$-graded $\kk$-algebra with decomposition
$A = \bigoplus_{\sigma \in G} A_\sigma$, and let $\c$ be an alternating bicharacter
defined on $G$.  We say that $A$ is \emph{$\c$-color commutative} (or simply
\emph{color commutative} if $\c$ is understood) if for every
$x \in A_\sigma$ and $y \in A_\tau$, one has $xy = \c(\sigma,\tau)yx$.
An element $x \in A_\sigma$ is said to be $G$-\emph{homogeneous}.
We call the $G$-degree of a $G$-homogeneous element $x$ the \emph{color} of $x$,
and we denote this by $\gdeg{x}$. We also refer to $G$ as the \emph{group of colors of $A$}. If $x$ and $y$ are $G$-homogeneous we abuse 
notation and use $\c(x,y)$ to denote $\c(\gdeg{x},\gdeg{y})$.

\subsection*{Color DG algebras}
Let $Q$ be a color commutative connected graded $\kk$-algebra and denote by $G$ its group of colors. Let $A$ be a DG $Q$-algebra.  We say that $A$ is a \emph{$\c$-color DG $Q$-algebra}
provided $A$ is $G$-graded with a grading compatible with the homological and internal grading of $A$,
and the differential on $A$ is $G$-homogeneous of color $e_G$.

We also assume that a color DG $Q$-algebra $A$ is \emph{graded color commutative}.  
That is, for all
$x,y \in A$, homogeneous with respect to all gradings, we assume that $xy = (-1)^{|x||y|}\c(x,y)yx$, and that $x^2 = 0$ when
$x$ is of odd homological degree.

\begin{chunk}  The color opposite of $A$ is the DG algebra $A^{\mathrm{op}}$ with the same underlying complex as $A$ and product given by
\[
a\cdot^{\mathrm{op}}b=(-1)^{|a||b|}\c(a,b)ba,
\]
where $a,b$ are elements of $A$ homogeneous with respect to all the gradings of $A$. By the color commutativity of $A$ it follows that $a\cdot^{\mathrm{op}}b=ab$, therefore $A^{\mathrm{op}}$ is just $A$.
The \emph{enveloping algebra} of $A$, denoted by $A^\e$, is the color DG $Q$-algebra with underlying complex $A\otimes_Q A$ and product given by
\[
(a\otimes b)(c\otimes d)=(-1)^{|b||c|}\c(b,c)ac\otimes bd.
\]
where $a,b,c,d$ are elements of $A$ homogeneous with respect to all the gradings of $A$.
\end{chunk}

\subsection*{Color DG modules}
Let $A$ be a graded color commutative DG $Q$-algebra, where $Q$ is a color commutative connected graded $\kk$-algebra. 
In this section we review conventions and terminology regarding color DG $A$-(bi)modules; see \cite[Section 4]{ColorDGA} for details.
 Throughout this article, a left (right) color DG $A$-module is a left (right) DG $A$-module equipped with a $G$-grading, where $G$ is the group of colors of $A$, compatible with respect to the homological and internal gradings.

\begin{chunk}\label{c:bimodulestructure}As usual a color DG $A$-bimodule is a left and right color DG $A$-module where the two actions are compatible; equivalently, it is a left color DG $A^\e$-module.  We say that 
a DG $A$-bimodule $M$ is \emph{symmetric} provided $am=(-1)^{|a||m|}\c(a,m)ma$
for all $m\in M,a\in A$ homogeneous with respect to all gradings.
Each right color DG $A$-module $M$ can be canonically made into a symmetric color bimodule by prescribing a left DG $A$-module with the following action:
\[
a\cdot m:=(-1)^{|a||m|}\c(a,m)ma
\]for each $a\in A$  and $m\in M$.  Similarly, we can switch from left color DG $A$-modules to right color DG $A$-modules according to the following rule: 
\[
m\cdot a:=(-1)^{|a||m|}\c(m,a)am
\]for each $a\in A$  and $m\in M$. Hence, each left color (or right) DG $A$-module can naturally be viewed as  a symmetric color DG $A$-bimodule according to this convention; therefore, we will not specify a side for color DG modules over color commutative algebras. 
\end{chunk}

\begin{chunk}\label{c:tensorproduct}
Let $M,N$ be color DG $A$-modules. The color DG $A$-modules $M\otimes_AN$ and $\Hom_A(M,N)$ are defined in the standard way, see \cite[Definition 4.6 and Definition 4.8]{ColorDGA}. The suspension of $M$, denoted by $\Sigma M$, is the color DG $A$-module with 
\[
(\mathsf{\Sigma} M)_i:=M_{i-1},  \ \  \del^{\mathsf{\Sigma}M}:=-\del^M, \   \text{ and }  \ a\cdot m:=(-1)^{|a|}am.
\] 

The homology of $M$, denoted by $\hh(M)$, is a color  $\hh(A)$-module over the color commutative $\hh_0(A)$-algebra $\hh(A).$ 
We say $M$ is  \emph{finite} provided that the color module $M^\natural$ is finitely generated over $A^\natural$ where $(-)^\natural$ is the forgetful functor to the category of color $A$-modules.  In particular, any bounded complex of finitely generated color $\hh_0(A)$-modules is a finite color DG $A$-module via restricting scalars along the augmentation $A\to \hh_0(A)$.
\end{chunk}

\begin{chunk}\label{c:enveloping}
 The \emph{diagonal} is the multiplication map 
$A^\e\to A$, which is a morphism of color DG $Q$-algebras. In the sequel, it is through the diagonal that $A$ will always be regarded as a color DG $A^\e$-module. 
\end{chunk}


\subsection*{Color DG homological algebra}

Let $A$ be a graded color commutative DG algebra.

\begin{chunk}
Let $F$ be a color DG $A$-module. We say that $F$ is \emph{semifree} if there exists a chain of color DG $A$-modules 
\[
0=F(-1)\subseteq F(0)\subseteq F(1)\subseteq F(2)\subseteq \ldots 
\] satisfying $\bigcup F(n)=F$ and for each $n\geq 0$ there is an isomorphism of color DG $A$-modules
\[
F(n)/F(n-1)\cong \bigoplus_{i\in \mathbb{Z}} \mathsf{\Sigma}^i A^{(X_{i,n})}
\] for some (possibly infinite) sets $X_{i,n}$. 
\end{chunk}

\begin{chunk}
\label{tensorSemifree} \label{homotopyEquiv}
A color DG $A$-module $P$ is said to be \emph{semiprojective} provided  $\Hom_A(P,-)$ preserves surjections and quasi-isomorphisms. 
Below are standard properties when $A$ has trivial color, so we leave the proofs to the  reader (cf. \cite[Chapter 6]{FHT} or \cite[Chapter 1]{IFR}).

From the definition it follows easily that a quasi-isomorphism between semiprojectives is in fact a homotopy equivalence. Also, any semifree DG $A$ module is semiprojective. Finally, if $P$ is semiprojective $P\otimes_A -$ preserves injections and quasi-isomorphisms. 
\end{chunk}
\begin{chunk}Let $P$ be a color DG $A$-module. We say that $P$ is \emph{perfect over Q} provided that $P$ is quasi-isomorphic to a bounded complex of finitely generated projective $Q$-modules; here we are regarding $P$ as a color DG $Q$-module via restriction of scalars along the structure map $Q\to A. $ We say $P$ is \emph{strongly perfect over Q} if it is itself a bounded complex of finitely generated projective $Q$-modules. 
\end{chunk}

\begin{chunk}\label{prop:DGModRes}

For each DG $A$-module $M$ there exists a surjective quasi-isomorphism of DG $A$-modules $F\xra{\simeq }M$ where $F$ is semifree. We call $F\xra{\simeq} M$ a \emph{semifree resolution of M over A}, and any two semifree resolutions of $M$ are homotopy equivalent. Furthermore, when $A$ is a nonnegatively graded DG algebra with $\hh(A)$ noetherian and $\hh(M)$  a noetherian  graded $\hh(A)$-module, there exists a semifree resolution $F\xra{\simeq}M$ such that
\[
F^\natural\cong \bigoplus_{j=i}^\infty \mathsf{\Sigma}^j (A^{\beta_j})^{\natural}
\] for a some fixed integer $i$ and nonnegative integers $\beta_j$ (see \cite[Proposition B.2]{AISS}).

 Mimicking the proofs in \cite[Proposition 6.6]{FHT} one can show the existence and uniqueness, up to homotopy, of semifree resolutons; one only need to mind the colors while adapting the arguments there.

Let $M,N$ be color DG $A$-modules, and let $F$ be a semifree resolution of $M$. It follows from the previous properties that the homology of $\Hom_A(F,N)$ and $F\otimes_AN$ does not depend on the choice of the resolution $F$. We define 
\[
\Ext_A(M,N):=\hh(\Hom_A(F,N)) \ \text{ and } \ \Tor^A(M,N)=\hh(F\otimes_A N)
\] for each color DG $A$-module $N$. These are naturally graded color modules over $\hh(A)$ (see \ref{c:tensorproduct} for the $\hh(A)$-structures).  
\end{chunk}

\begin{chunk}\label{adjoinvariables}
Let $A$ be a color DG algebra over a color commutative ring. Let $z$ be a cycle of $A$ homogeneous with respect to all gradings. There is a semifree extension of $A$, denoted by $\Ay{A}{y\mid \del y=z}$,   where $z$ becomes a boundary. This extension is constructed by  adjoining the variable  $y$ which is an skew exterior variable when $z$ has even homological degree or a skew divided power variable when $z$ has odd homological degree; in either case, the variable $y$ color commutes with $A$. It is straightforward to see that $A$ is a DG subalgebra of $\Ay{A}{y}$. Hence,  we can inductively adjoin variables to kill cycles:
\[
\Ay{A}{y_1,\ldots,y_m\mid \del y_i=z_i}:=\Ay{\Ay{A}{y_1,\ldots,y_{m-1}\mid \del y_i=z_i}}{y_m\mid \del y_m=z_m}.
\]
See \cite[Proposition 2.5 and Proposition 2.7]{ColorDGA} for further details.
\end{chunk}

The final two subsections provide context  for the  the calculations in \Cref{s:HH}, and explain the choice of terminology regarding the operators discussed in \Cref{section:severalactions}.

\subsection*{Braided tensor categories, general case}

In this subsection we recall some constructions from \cite{braided}.   
A monoidal category $(\mathcal{C},\otimes)$ is said to be \emph{braided} if for all objects $U$ and $V$ of  $\mathcal{C}$ there exist functorial isomorphisms
\[
R_{U,V}:U\otimes V\rightarrow V\otimes U,
\]
satisfying the \emph{hexagon axioms}, see \cite[Section 2]{braided}. If, furthermore, $R_{U,V}R_{V,U}=id_{V\otimes U}$ for all objects in $\mathcal{C}$, then $\mathcal{C}$ is called a \emph{symmetric monoidal category}.

In \cite{braided}, Baez defines braided tensor categories using commutative rings, we work with the more general notion of color commutative rings; the proofs of the statements below remain unchanged under this more general hypothesis. Given a color commutative ring $\mathbbm{F}$, we define a \emph{braided (symmetric) tensor category of} $\mathbbm{F}$-\emph{modules} to be a braided (symmetric) monoidal category $\mathcal{C}$ equipped with a faithful functor $F$ to the category of color $\mathbbm{F}$-modules, satisfying a list of axioms that the interested reader can find in \emph{loc. cit.}. Therefore for the remainder of this subsection, we will work in a fixed braided tensor category $\mathscr{V}$ of $\mathbbm{F}$-modules, where $\mathbbm{F}$ is some color commutative ring. As in \cite{braided}, we will identify objects and morphisms in $\mathscr{V}$ with their images under $F$.

Let $A$ be an algebra object in $\mathscr{V}$, with an associative multiplication given by $m_A$. Let $A^{\mathrm{op}}$ denote $A$ with the multiplication map $m_AR_{A,A}$. It is proved in \cite[Lemma 1]{braided}, that $A^{\mathrm{op}}$ is an algebra in $\mathscr{V}$. Let $B$ be another algebra in $\mathscr{V}$ with product $m_B$, then  $A\otimes B$ is an algebra in $\mathscr{V}$ with multiplication given by $(m_A\otimes m_B)(id_A\otimes R_{B,A}\otimes id_B)$ (cf.  \cite[Lemma 2]{braided}).  The enveloping algebra of $A$, denoted by $A^\e$, is $A\otimes A^{\mathrm{op}}$.


Assume that $\mathscr{V}$ is a symmetric tensor category. 
\begin{enumerate}
\item \cite[Lemma 4]{braided} If $A$ is an algebra object in $\mathscr{V}$, then $A$ is a left $A^\e$-module in $\mathscr{V}$ with action given by $m_A(id_A\otimes m_AR_{A,A})$.
\item \cite[Lemma 5]{braided} If $A$ is an algebra object in $\mathscr{V}$, then $A$ is a right $A^\e$-module in $\mathscr{V}$ with action given by $m_AR_{A,A}(m_A\otimes id_A)$.
\end{enumerate}
These facts are applied in \cite[Section 3]{braided} to define the braided Hochschild homology of an algebra object $A$ in $\mathscr{V}$, whenever $A$ is flat over $\mathbbm{F}$. We define a dual notion below. 
\begin{definition}Let $\mathbbm{F}$ be a color commutative ring, and let $\mathscr{V}$ be a symmetric tensor category of $\mathbbm{F}$-modules. Let $A$ be an algebra object in $\mathscr{V}$ which is projective over $\mathbbm{F}$. The \emph{braided Hochschild cohomology} of $A$ with respect to $\mathbb{F}$ is
\[
\mathrm{HH}(A|\mathbbm{F})=\Ext_{A^\e}(A,A),
\]
where the right-hand side is the homology of the Hom complex of right linear maps  $\mathrm{Hom}_{A^e}(B,B)$; here  $B$ is the bar resolution constructed in \cite[Theorem 1]{braided} and  $\Ext_{A^\e}(A,A)$ has the  composition product.
\end{definition}

\subsection*{Symmetric tensor categories, color commutative case}
Let $Q$ be a color commutative ring, and let $\mathcal{C}$ be the category of color DG $Q$-modules. The tensor product $\otimes_Q$, defined in \cite[Definition 4.6]{ColorDGA}, gives $\mathcal{C}$ the structure of a monoidal category. Let $\c$ be the bicharacter associated to $Q$. For every pair of objects $U,V\in\mathcal{C}$, we define maps $R_{U,V}:U\otimes_QV\rightarrow V\otimes_QU$ by
\[
u\otimes v\mapsto(-1)^{|u||v|}\c(u,v)v\otimes u,
\]
for all $u\in U$ and $v\in V$ homogeneous elements with respect to all gradings. These maps give $\mathcal{C}$ the structure of a symmetric monoidal category. The natural embedding of $\mathcal{C}$ into the category of color $Q$-modules gives $\mathcal{C}$ the structure of a symmetric tensor category, we denote it by $\mathscr{V}$.

Let $A$ be a color commutative DG $Q$-algebra in $\mathscr{V}$. The product on $A^{\mathrm{op}}$, defined by $m_AR_{A,A}$, simplifies to
\[
a\cdot^{\mathrm{op}}b=(-1)^{|a||b|}\c(a,b)ba=ab,
\]
where $a,b$ are homogeneous elements with respect to all the gradings.
This shows that $A^{\mathrm{op}}$ is isomorphic to $A$ as a color DG $Q$-algebra. 

Let $A$ and $B$ be color commutative algebras in $\mathscr{V}$. The multiplication on the tensor product $A\otimes B$ simplifies to
\[
(a\otimes b)(a'\otimes b')=(-1)^{|b||a'|}\c(b,a')(aa')\otimes(bb'),
\]
where $a,b,a',b'$ are homogeneous elements with respect to all the gradings.

The enveloping algebra $A^\e$ is the tensor product $A\otimes_QA^{\mathrm{op}}=A\otimes_QA$. It follows from \cite[Lemma 3]{braided}, or from a straightforward check, that $A^\e$ is a color commutative DG $Q$-algebra.

The left and right $A^\e$ action on $A$, given by \cite[Lemma 4 \& Lemma 5]{braided}, simplify to 
\[
(a\otimes b)\cdot c=(-1)^{|b||c|}\c(b,c) acb,\quad c\cdot(a\otimes b)=(-1)^{|a||c|}\c(c,a)acb,
\]
for $a,b,c\in A$, homogeneous elements with respect to all gradings.

Therefore,  the braided Hochschild cohomology of $A$ relative to  $Q$ is 
\[
\mathrm{HH}(A|Q)=\mathrm{Ext}_{A^\e}(A,A),
\]
provided that $A$ is semiprojective over $Q$.

If $Q$ is a skew polynomial ring and $R$ a quotient of $Q$ by an ideal generated by normal elements, then there exists a color DG $Q$-algebra resolution of $R$, obtained via the process of killing cycles illustrated in \cref{adjoinvariables}, see also \cite[Section 2]{ColorDGA}. We denote this DG algebra resolution by $E$. As in \cite[Remark 3.3]{shukla} we make the following
\begin{definition}\label{defderivedbraided}
Let $Q$ be a skew polynomial ring and $R$ a quotient of $Q$ by an ideal generated by normal elements. Let $E$ be a color DG $Q$-algebra resolution of $R$. The \emph{derived braided Hochschild cohomology} of $R$ relative to $Q$ is
\[
\dHH{R}{Q}=\hh\hh(E|Q).
\]
\end{definition}
Observe that
\[
\dHH{R}{Q}=\Ext_{R\otimes_Q^{\mathbf{L}}R}(R,R),
\]
which justifies the use of the adjective \emph{derived} in the definition of $\dHH{R}{Q}$.
\begin{remark}
Braided Hochschild homology was first introduced in \cite{braided} by Baez.  Derived Hochschild cohomology was defined by MacLane in \cite{MacLane} for $\mathbb{Z}$-algebras and later generalized by Shukla in \cite{OriginalShukla} for algebras over general commutative noetherian rings. Quillen, in \cite{Quillen}, recognized it as a derived version of Hochschild cohomology; \cite{shukla} notes that derived Hochschild cohomology is also known as \emph{Shukla cohomology}. For a comparison of Hochschild cohomology and derived Hochschild cohomology see \cite{comparison}.
\end{remark}

\section{Semifree resolution of the diagonal}\label{s:semifreediagonal}

Throughout this paper we fix the following notation.  Let $Q=\kk_\mathfrak{q}[x_1,\ldots,x_n]$ be a skew polynomial ring, where $\kk$ is a field and $\mathfrak{q}=(q_{i,j})$ is a matrix with invertible entries such that $q_{i,j}=q_{j,i}^{-1}$ for all $i,j$ and $q_{i,i}=1$ for all $i$. The ring $Q$ can be regarded as a color commutative ring in the standard way, see \cite[Example 3.2]{ColorDGA}, denote by $G$ its group of colors. We fix  a sequence of normal elements  $\f=f_1,\ldots,f_c$  in  $(x_1,\ldots,x_n)^2$ and set  $R$ to be the quotient $R= Q/(\f)$
where we regard $R$ as a color commutative ring with the same group of colors as $Q$. 
Finally, let  \[E=\Ay{Q}{e_1,\ldots,e_c\mid\partial(e_i)=f_i}\] be the (skew) Koszul complex on $\f$ over $Q$ (cf. \Cref{adjoinvariables}).

\begin{chunk}Recall  from \ref{c:enveloping},  $E$ is regarded as a color DG $E^\e$-algebra via the diagonal $E^\e\to E$ given by 
\[a\otimes b\mapsto ab.\]
There is an isomorphism of color DG $Q$-algebras 
\[
E^\e\cong \Ay{Q}{e_1,\ldots,e_c,e_1',\ldots,e_c'\mid \partial(e_i)=\partial(e_i')=f_i},
\]
that identifies    $1\otimes e_i$ and $e_i\otimes 1$ with  $e_i$ and $e_i'$, respectively. Hence, we will freely identify these models without further mention in the rest of the article. 
\end{chunk}

\begin{construction}\label{semifreediagonalres}
Notice that the  elements $e_i'-e_i$ are $G$-homogeneous cycles in the color DG $Q$-algebra $E^\e$. Moreover, 
they pairwise skew commute and square to zero and so by \cite[Proposition 2.9]{ColorDGA}, they are killable. Therefore, we introduce a set of (skew) divided power variables 
$Y=\{y_1,\ldots, y_c\}$ of homological degree 2 and define
\[
\Ay{E^\e}{Y}:= \Ay{E^\e}{Y\mid\del y_i=e_i'-e_i}.
\] Note that $\Ay{E^\e}{Y}$ is a color DG $E^\e$-algebra and is equipped with a morphism of color DG $Q$-algebras 
$\epsilon: \Ay{E^\e}{Y}\to E$ given by 
\[
(a\otimes b)y_1^{(h_1)}y_2^{(h_2)}\ldots y_c^{(h_c)}\mapsto \begin{cases}
 0 & \text{if any }h_i>0\\
 ab & \text{otherwise}
\end{cases}.
\] Moreover, the morphism is compatible with the $E^\e$-actions meaning that we have the following commutative diagram of color DG $E^\e$-algebras where the unlabeled maps are the canonical ones
\begin{center}
\begin{tikzcd}
& \Ay{E^\e}{Y} \arrow[dr,"\epsilon"] & \\
E^\e \arrow{ur} \arrow{rr} & & E
\end{tikzcd}
\end{center}
\end{construction}

\begin{theorem}\label{prop:Eres}
The   morphism of DG $E^\e$-algebras $\epsilon: \Ay{E^\e}{Y}\to E$ defined in Construction \ref{semifreediagonalres} is a semifree resolution of $E$ over $E^\e$. 
\end{theorem}
\begin{proof}Clearly $E^\e\langle Y\rangle$ is a semifree DG $E^\e$-algebra. So it suffices to check that  $\epsilon$ is a quasi-isomorphism. 

First, let $\bm{x}=x_1,\ldots,x_c$ be (skew) exterior variables of homological degree 1 and such that $\gdeg{x_i}=\gdeg{e_i}$. Set \[Q\langle \bm{x}\rangle=Q\langle \bm{x} \mid \del x_i=0\rangle;\] that is, the (skew) exterior algebra on $\bigoplus_{i=1}^cQx_i$ and let $\pi: Q\langle \bm{x}\rangle\to Q$ be the canonical augmentation map. There is an evident isomorphism of color DG $Q$-algebras
$\varphi: E^\e\to Q\langle \bm{x}\rangle\otimes_Q E$ determined by 
\[
e_i\otimes 1\mapsto  1\otimes e_i  \text{ and }1\otimes e_i\mapsto 1\otimes e_i-x_i\otimes 1.
\] Furthermore $\varphi$ is compatible with the augmentation maps to $E$, meaning the following diagram of color DG $Q$-algebras commutes
\begin{center}
\begin{tikzcd}
E^\e \arrow[d] \arrow[r,"\varphi"]&  Q\langle \bm{x}\rangle\otimes_Q E\arrow[d,"\pi\otimes1"] \\
E  \arrow[r,"="] &  E.
\end{tikzcd}
\end{center} By \cite[Lemma 6.4]{ColorDGA},  $\varphi$ extends to an morphism of color DG $Q$-algebras  \[
\tilde{\varphi}: E^\e\langle Y\rangle \to (Q\langle \bm{x}\rangle\otimes_Q E)\langle Y\mid \del y_i=\varphi(e_i\otimes1-1\otimes e_i)\rangle;
\] as $\tilde{\varphi}$ maps a graded $Q$-basis of $E^\e$ to a graded $Q$-basis of $Q\langle \bm{x}\rangle\otimes_Q E$, we conclude that $\tilde{\varphi}$ is an isomorphism of color DG $Q$-algebras. Now note $\varphi(e_i\otimes1-1\otimes e_i)=x_i\otimes 1$ and hence, there is an isomorphism $\psi$
\[
(Q\langle \bm{x}\rangle\otimes_Q E)\langle Y\mid \del y_i=\varphi(e_i\otimes1-1\otimes e_i)\rangle\cong E\otimes _Q Q\langle \bm{x}, Y\mid \del x_i=0, \ \del y_i=x_i\rangle
\] as augmented, to $E$, DG $Q$-algebras. Finally, it follows from \cite[Theorem 2.15]{ColorDGA} that \[Q\langle \bm{x}, Y\mid \del x_i=0, \ \del y_i=x_i\rangle\xra{\simeq} Q
\] and so the commutative diagram
\begin{center}
\begin{tikzcd}
E^\e \langle Y\rangle\arrow[d] \arrow[r,"\psi \tilde{\varphi}"]& E\otimes_Q Q\langle \bm{x},Y\rangle\arrow[d,"\simeq"] \\
E  \arrow[r,"="] &  E
\end{tikzcd}
\end{center} completes the proof. 
\end{proof}
 
\begin{notation}
For the rest of the paper $\Ay{E^\e}{Y}$ will denote the resolution of $E$ from \cref{semifreediagonalres}.
\end{notation}

\section{Derived braided Hochschild cohomology}\label{s:HH}

Directly from \Cref{prop:Eres},
\[
\Tor^{E^\e}(E,E)\cong \hh(\Ay{E^\e}{Y}\otimes_{E^\e}E).
\] It is clear that there is an isomorphism of color DG $Q$-algebras
\[
\Ay{E^\e}{Y}\otimes_{E^\e}E\cong E\langle Y\mid \del y_i=0\rangle.
\] Therefore, the \emph{braided Hochschild homology} of $E$ over $Q$ is the graded color $Q$-algebra 
\[
\hh(E)\otimes_Q Q\langle Y\rangle. 
\]
This section is devoted to the calculation of  Hochschild \emph{co}homology, $\hh\hh(E|Q)=\Ext_{E^\e}(E,E)$,  as a graded color $Q$-algebra.

\begin{chunk} \label{LDGlie}
 Let $A$ be a DG $R$-algebra and let $\Ay{A}{Y}$ be a semifree extension of $A$. We let
 $D:=\Der_A(\Ay{A}{Y},\Ay{A}{Y})$ denote the subset of $\Hom_A(\Ay{A}{Y},\Ay{A}{Y})$ consisting of $A$-linear color derivations, see \cite[Definition 5.1]{ColorDGA}. It is straightforward to check that $D$ is a subcomplex of $\Hom_A(\Ay{A}{Y},\Ay{A}{Y})$. Equipping it with the following bracket and squaring operations 
\[
[\theta,\xi]=\theta\xi-(-1)^{|\theta||\xi|}\c(\theta,\xi)\xi\theta
\]
\[
\zeta^{[2]}=\zeta^2,
\]
where $\theta,\xi,\zeta\in\Der_A(\Ay{A}{Y},\Ay{A}{Y})$ are homogeneous with respect to all the gradings and $\zeta$ has odd homological degree,  makes $D$ a color DG Lie $A^\natural$-algebra, see \cite[Definition 7.1 and 7.9]{ColorDGA} for the definition of color DG Lie algebra. The proof is essentially contained in \cite[Lemma 7.10]{ColorDGA}, where it is proved when $A=R$.  The same proof works  in this more general case; we show that the bracket is $A^\natural$-bilinear. Let $a\in A,\theta,\xi\in\Der_A(\Ay{A}{Y},\Ay{A}{Y})$ be homogeneous with respect to all the gradings, then
\begin{align*}
[a\theta,\xi]&=a\theta\xi-(-1)^{|a\theta||\xi|}\c(a\theta,\xi)\xi(a\theta)\\
&=a\theta\xi-(-1)^{|a||\xi|+|\theta||\xi|}\c(a,\xi)\c(\theta,\xi)(-1)^{|\xi||a|}\c(\xi,a)a\xi\theta\\
&=a[\theta,\xi].
\end{align*}
The remaining checks are similar.
\end{chunk}

\begin{chunk} Let $A$ be a color DG algebra. Let $g_1,\ldots, g_c$ be elements of the group of colors of $A$, let $n_1,\ldots,n_c$ be integers, and let $z_1,\ldots, z_c$ be indeterminates, we denote by
\[
A[z_1,\ldots,z_c\mid \gdeg{z_i}=g_i, |z_i|=n_i, i=1,\ldots,c]
\]
the Ore extension obtained by adding the indeterminates $z_1,\ldots, z_c$ to $A$, satisfying the following commuting relations:
\[
z_is=(-1)^{n_i|s|}\c(g_i,s)sz_i,\quad\mathrm{for\;all}\;s\in A,
\]
\[
z_iz_j=(-1)^{n_in_j}\c(g_i,g_j)z_jz_i,\quad\mathrm{for\;all}\;i,j=1,\ldots,c,
\]
where $s$ is homogeneous with respect to all the gradings of $A$.
\end{chunk}

We use $\Ay{E^\e}{Y}$ to compute the   derived braided Hochschild cohomology of $E$ relative to $Q$ (cf. \Cref{defderivedbraided}).

\begin{theorem}
There is an isomorphism of graded color $\hh(E^\e)$-algebras  
\[
\hh\hh(E|Q)\cong  \hh(E) [\chi_1,\ldots,\chi_c\mid\gdeg{\chi_i}=\gdeg{f_i}^{-1},|\chi_i|=2,i=1,\ldots,c].
\]
Moreover the variables $\chi_1,\ldots,\chi_c$ correspond to the homology classes of derivations $\theta_1,\ldots,\theta_c\in\Der_{E^\e}(\Ay{E^\e}{Y},\Ay{E^\e}{Y})$, satisfying $\theta_i(y_j)=\delta_{i,j}$.
\label{theoremhhcomputed}
\end{theorem}

\begin{proof}

We consider the following diagram 
\begin{center}
\begin{tikzpicture}[baseline=(current  bounding  box.center)]
 \matrix (m) [matrix of math nodes,row sep=3em,column sep=4em,minimum width=2em] {
\Der_{E^\e}(\Ay{E^\e}{Y},\Ay{E^\e}{Y})&\Hom_{E^\e}(\Ay{E^\e}{Y},\Ay{E^\e}{Y})\\
\Der_{E^\e}(\Ay{E^\e}{Y},E)&\Hom_{E^\e}(\Ay{E^\e}{Y},E)\\
\Hom_{\Ay{E^\e}{Y}}(\mathrm{Diff}_{E^\e}\;\Ay{E^\e}{Y},E)&\Hom_E(\Ay{E^\e}{Y}\otimes_{E^\e}E,E)\\
\Hom_{E}((\mathrm{Diff}_{E^\e}\;\Ay{E^\e}{Y})\otimes_{\Ay{E^\e}{Y}}E,E)&\\
\Hom_E(EY,E)&\Hom_E(\Ay{E}{Y\mid \partial y_i=0\;\forall i},E)\\
E\chi_1\oplus\cdots\oplus E\chi_c&E[\chi_1,\ldots,\chi_c\mid\gdeg{\chi_i}=\gdeg{f_i}^{-1},|\chi_i|=2]\\};
\path[->] (m-1-1) edge (m-1-2);
\path[->] (m-2-1) edge  (m-2-2);
\path[->] (m-1-1) edge node[right] {\cite[Corollary 5.3]{ColorDGA}} (m-2-1);
\path[->] (m-1-2) edge node[left]{\cite[Proposition 4.12]{ColorDGA}}(m-2-2);
\path[->] (m-2-1) edge node[right]{\cite[Proposition 5.2]{ColorDGA}} (m-3-1);
\path[->] (m-2-2) edge node[left]{\cite[Proposition 4.9]{ColorDGA}} (m-3-2);
\path[->] (m-3-2) edge (m-5-2);
\path[->] (m-3-1) edge node[right]{\cite[Proposition 4.9]{ColorDGA}} (m-4-1);
\path[->] (m-4-1) edge (m-5-1);
\path[->] (m-6-1) edge (m-6-2);
\path[->] (m-5-1) edge (m-6-1);
\path[->] (m-5-2) edge (m-6-2);
\end{tikzpicture}
\end{center}
where the horizontal maps are inclusions, the vertical maps are quasi-isomorphisms, $E\chi_1\oplus\cdots\oplus E\chi_c$ is the free DG $E$-module with basis $\chi_1,\ldots,\chi_c$ and such that the differential of the elements $\chi_1,\ldots,\chi_c$ is zero and the differential on the indeterminates of  $E[\chi_1,\ldots,\chi_c\mid\gdeg{\chi_i}=\gdeg{f_i}^{-1},|\chi_i|=2]$ is trivial. The isomorphism $(\mathrm{Diff}_{E^\e}\;\Ay{E^\e}{Y})\otimes_{\Ay{E^\e}{Y}}E\cong EY$, where $EY$ is the free $E$-module with basis $Y$ and trivial differential on the elements of $Y$,  follows directly from the construction of the module of differentials. The isomorphism $\Ay{E^\e}{Y}\otimes_{E^\e}E\cong\Ay{E}{Y\mid\partial y_i=0\;\forall i}$ follows from the construction of the resolution $\Ay{E^\e}{Y}$. The bottom left vertical map is the map $\theta\mapsto\sum_{i=1}^c\theta(y_i)\chi_i$ and similarly for the bottom right vertical map. It is straightforward to check that this is a commutative diagram and that $\chi_i\in E\chi_1\oplus\cdots\oplus E\chi_c$ corresponds to the derivation $\theta\in\Der_{E^\e}(\Ay{E^\e}{Y},\Ay{E^\e}{Y})$ such that $\theta(y_j)=\delta_{i,j}$.

For the remainder of the proof  $L$ will denote the homology of the complex $\Der_{E^\e}(\Ay{E^\e}{Y},\Ay{E^\e}{Y})$; from \Cref{LDGlie} it follows that $L$ is a graded color Lie $\hh(E^\e)$-algebra. The map induced in homology by the bottom map is the inclusion of $\hh(E^\e)$-modules $\hh(E)\chi_1\oplus\cdots\oplus \hh(E)\chi_c\rightarrow \hh(E)[\chi_1,\ldots,\chi_c\mid\gdeg{\chi_i}=\gdeg{f_i}^{-1},|\chi_i|=2]$. Therefore the top map induces, in homology, an injection of graded color Lie $\hh(E^\e)$-algebras $\iota: L\rightarrow\mathrm{Lie}(\hh\hh(E|Q))$. Let $\iota': UL\rightarrow\hh\hh(E|Q)$ be the universal extension of $\iota$, where $UL$ is the universal enveloping algebra of $L$ (see \cite[Definition 7.6 and Remark 7.7]{ColorDGA}), we claim that $\iota'$ is an isomorphism of associative $\hh(E^\e)$-algebras. Indeed, since $L=\hh(E)\chi_1\oplus\cdots\oplus \hh(E)\chi_c$ and $\gdeg{\chi_i}=\gdeg{f_i}^{-1}$ for all $i$, it follows that $UL\cong \hh(E)[\chi_1,\ldots,\chi_c\mid\gdeg{\chi_i}=\gdeg{f_i}^{-1},|\chi_i|=2]$. The computations performed at the beginning of this proof show that $\hh\hh(E|Q)$ is isomorphic to $\hh(E)[\chi_1,\ldots,\chi_c\mid\gdeg{\chi_i}=\gdeg{f_i}^{-1},|\chi_i|=2]$ as a $\hh(E)$-module, and they also show that the map $\iota'$ is surjective. In each homological degree $\iota'$ is a surjective map of free $\hh(E)$-modules of the same rank, and since $\hh(E)$ is noetherian it follows that $\iota'$ must be injective in each homological degree, and therefore it is itself injective.
\end{proof}

\begin{corollary}
When $\f$ is a regular sequence, there is an isomorphism of graded color $R$-algebras   \[\dHH{R}{Q}\cong R[\chi_1,\ldots,\chi_c\mid\gdeg{\chi_i}=\gdeg{f_i}^{-1},|\chi_i|=2,i=1,\ldots,c].\] 
\label{t:2}
\end{corollary}
\begin{proof}
This follows from the previous theorem since
\[
\hh\hh(E|Q)=\dHH{R}{Q} \ \text{ and } \  \hh(E)=R;
\]
therefore the isomorphism claimed in the Corollary is one of $\hh(E^\e)=\mathrm{Tor}^Q(R,R)$-algebras, and therefore one of $R$-algebras.
\end{proof}

\begin{remark}
As a consequence of Theorem \ref{theoremhhcomputed}, \[\c(\chi_i,-)=\c(y_i,-)^{-1}=\c(f_i,-)^{-1}.\] In the sequel we will make use of this fact in Section \ref{s:extmodules} to define a color DG $S$-module whose homology computes $\Ext$ over a skew complete intersection where 
\[
S=Q[\chi_1,\ldots,\chi_c\mid\gdeg{\chi_i}=\gdeg{f_i}^{-1},|\chi_i|=2,i=1,\ldots,c].
\] 
\end{remark}

\section{Products on the derived braided Hochschild cohomology}\label{section:severalactions}

In this section we clarify that for a skew complete intersection $R=Q/(\f)$ the composition action of  $\dHH{R}{Q}$ on $\Ext_R(M,N)$ can be computed at the chain level with the so-called ``cup product," see \Cref{cupdef}. The first action establishes properties expected from a well-posed support theory, see \Cref{p:basicprops}; the latter allows us to prove the intersection formula in  \Cref{t:symmetry}.

We fix the following notation. Let $A$ be a color DG $Q$-algebra. Let  $\varepsilon: P\xra{\simeq} A $ be a semiprojective resolution of $A$ over $A^\e$. 

Recall that for any color DG $A^\e$-module $X$,  $\Hom_{A^\e}(P,X)$ is naturally a right DG $\Hom_{A^\e}(P,P)$-module via right composition of maps. This action induces a graded right $\HH(A|Q)$-module structure on $\Ext_{A^\e}(A,X)$. First, we define a different right DG $\Hom_{A^\e}(P,P)$-module structure on $\Hom_{A^\e}(P,X)$, provided that $P$ admits a diagonal approximation; furthermore, we show that these two structures are the same in homology. This will be a slight generalization of \cite[Proposition 4.5]{BR}. Namely, we adapt \cite[Proposition 4.5]{BR} to the color setting, while also discussing  actions on certain Ext-modules.

\begin{definition}
 A \emph{co-unital diagonal approximation} for $\varepsilon: P\xra{\simeq} A$ is a morphism $\Phi:P\rightarrow P\otimes_A P$ of left color DG $A^\e$-modules so that 
\begin{center}
\begin{tikzcd}
            & P \arrow[ld, Rightarrow, no head] \arrow[rd, Rightarrow, no head] \arrow[d, "\Phi"]       &             \\
P & P\otimes_AP \arrow[l, "\varepsilon\otimes 1"] \arrow[r, "1\otimes\varepsilon"'] & P
\end{tikzcd}
\end{center}
\noindent
commutes, where we have identified $P\otimes_AA\cong P\cong A\otimes_AP$ with the appropriate multiplication maps. 
\end{definition}

\begin{definition}\label{cupdef}
Given a co-unital diagonal approximation $\Phi:P\rightarrow P\otimes_AP$ for $\epsilon$ we define a right \emph{cup action} $\cup_\Phi$ (or simply $\cup$) of $\Hom_{A^\e}(P,P)$ on $\Hom_{A^\e}(P,X)$ as
\[
\sigma\cup\tau=\mu(\sigma\otimes(\varepsilon\tau))\Phi:P\xra{\Phi}P\otimes_AP\xra{\sigma\otimes(\varepsilon\tau)}X\otimes_AA\xra{\mu}X,
\]
where $\sigma\in\Hom_{A^\e}(P,X),\tau\in\Hom_{A^\e}(P,P)$, and $\mu$ is the right multiplication induced from the $A^\e$-module structure of $X$.
\end{definition}

\begin{chunk}\label{augmentedcup}
Since $P$ is semiprojective, the surjective quasi-isomorphism $\varepsilon: P\xra{\simeq} A$  induces a surjective quasi-isomorphism 
\[
\varepsilon_*: \Hom_{A^\e}(P,P)\xra{\simeq}\Hom_{A}(P,A),
\]
 and so $\varepsilon_*$ prescribes a right action of $\Hom_{A^\e}(P,A)$ on $\Hom_{A^\e}(P,X)$. Indeed, for $\sigma\in \Hom_{A^\e}(P,X)$ and $\tau\in \Hom_{A^\e}(P,A)$ we first lift $\tau$ to $\tilde{\tau}$ in $\Hom_{A^\e}(P,P)$ satisfying $\varepsilon \tilde{\tau}=\tau.$ Now we define 
\[
\sigma\cdot\tau:=\sigma \cup \tilde{\tau}.
\] From the definition of $\cup$, it follows immediately that this action is independent of the choice of lifting $\tilde{\tau}$. 

\end{chunk}

\begin{proposition}\label{prop:cupprod}
Let $A$ be a color DG $Q$-algebra, let $\varepsilon:P\rightarrow A$ be a semiprojective  resolution of $A$ over $A^\e$, let $X$ be a color DG $A^\e$-module. If $P$ admits a co-unital diagonal approximation then, the cup and composition actions of $\Ext_{A^\e}(A,A)$ on $\Ext_{A^\e}(A,X)$ coincide.
\end{proposition}
\begin{proof}
First, we claim that $1\otimes \varepsilon$ and  $\varepsilon\otimes 1$ are homotopic in $\Hom_{A^\e}(P\otimes_A P, P).$ Indeed, as $P\otimes_A P$ is semiprojective over $A^\e$, $\varepsilon$ induces a quasi-isomorphism
\[
\Hom_{A^\e}(P\otimes_A P, P)\xra{\simeq }\Hom_{A^\e}(P\otimes_A P, A)
\] given by  $\psi\mapsto \varepsilon \psi.$  Note that $\varepsilon \circ (1\otimes \varepsilon)= \varepsilon \circ  (\varepsilon \otimes 1)$ and so the quasi-isomorphism above implies that $1\otimes \varepsilon- \varepsilon\otimes 1$ is a boundary in $\Hom_{A^\e}(P\otimes_A P,P), $ as needed. 

Now by the claim above, the square below commutes up to homotopy and hence, we have the following homotopy commutative diagram of DG $A^\e$-modules: 
\begin{center}
\begin{tikzcd}
P \arrow[r, "\Phi"] \arrow[rd, Rightarrow,no head]& P\otimes_AP \arrow[r, "1\otimes \tau"] \arrow[d, "\varepsilon\otimes1"'] & P\otimes_AP \arrow[r, "\sigma\otimes1"] \arrow[d, "1\otimes\varepsilon"'] & X\otimes_A P \arrow[r, "1\otimes\varepsilon"] & X\otimes_A A \arrow[dd, "\mu"] \\
& P \arrow[r,"\tau"]\arrow[rrrd,"\sigma\tau"']&P\arrow[rrd, "\sigma"]&&\\
&&&& X                             
\end{tikzcd}
\end{center}
where $\sigma\in\Hom_{A^\e}(P,X)$ and $\tau\in\Hom_{A^\e}(P,P)$. Hence $\sigma\cup \tau$, which is given by the composition of the four horizontal maps at the top, is homotopic to $\sigma\tau$.
\end{proof}

Now we specialize and  return to the setting of  \Cref{s:semifreediagonal}. In the proof of the following proposition we will make use of the divided powers DG $E^\e$-structure on $\Ay{E^\e}{Y};$ see \cite[Section 6]{ColorDGA} or \cite[Chapter 1]{GulLev} for more details regarding DG algebras with divided powers. 
\begin{proposition}\label{prop:approx}
The resolution $\Ay{E^\e}{Y}\xra{\simeq} E$ admits a co-unital diagonal approximation $\Phi:\Ay{E^\e}{Y}\rightarrow \Ay{E^\e}{Y}\otimes_E\Ay{E^\e}{Y}$, determined by $y\mapsto y\otimes1+1\otimes y$ for all $y\in Y$.
\end{proposition}
\begin{proof}
Let $\varphi:  E^\e\to E^\e\otimes_E E^\e $ be the color DG $Q$-algebra map given by 
\[
a\otimes b\mapsto (a\otimes 1)\otimes (1\otimes b). 
\] Notice that $\varphi$ is compatible with the system of divided powers on $E^\e$ and  $E^\e\otimes_E E^\e$. 
Let $y\in Y$ and let $e\in E$ be such that $\partial(y)=e\otimes1-1\otimes e$ in $\Ay{E^\e}{Y}.$ Observe that in $\Ay{E^\e}{Y}\otimes_E \Ay{E^\e}{Y}$ the following holds:
\begin{align*}
\varphi(\partial(y))&=\varphi(e\otimes 1-1\otimes e)\\&=(e\otimes1)\otimes(1\otimes 1)- (1\otimes1)\otimes(1\otimes e) \\
&=(e\otimes1)\otimes(1\otimes 1)- (1\otimes e)\otimes(1\otimes 1)\\
&   \hphantom{=} +(1\otimes 1)\otimes(e\otimes 1)- (1\otimes1)\otimes(1\otimes e) \\
&=(e\otimes 1-1\otimes e)\otimes (1\otimes 1)+(1\otimes 1)\otimes (e\otimes 1-1\otimes e)\\
&=\del(y)\otimes (1\otimes 1)+(1\otimes 1)\otimes \del(y).
\end{align*} 
By \cite[Lemma 6.4]{ColorDGA}, $\varphi$ uniquely extends to a morphism of color DG $E^\e$-algebras with divided powers  \[ \Phi: \Ay{E^\e}{Y}\to \Ay{E^\e}{Y}\otimes_E \Ay{E^\e}{Y}
\]
given by $
y\mapsto y\otimes 1+1\otimes y.
$ Finally, the co-unital approximation condition is immediate.
\end{proof}

\begin{chunk}\label{hochschildsameascup}
It follows from \cref{prop:cupprod} and \cref{prop:approx} that the composition action of $\dHH{R}{Q}$ on $\Ext_{E^\e}(E,X)$ is the same as the cup action for any color DG $E^\e$-module $X$.
\end{chunk}

\begin{chunk}\label{chunkPhi} Since $\Phi$ is a map of divided powers algebras, for any $y\in Y$ we have that 
\[
\Phi(y^{(H)})=\Phi(y_1)^{(h_1)}\ldots \Phi(y_c)^{(h_c)}
\] where $H=(h_1,\ldots,h_c)\in \mathbb{N}^c$. Now a computation shows that
\begin{equation}
    \label{equationforPhi} \Phi(y^{(H)})=\sum_{H'+H''=H}\left(\prod_{i<j}\c(y_i,y_j)^{h'_jh''_i}\right)y^{(H')}\otimes y^{(H'')},
\end{equation}
where $H'=(h_1',\ldots,h_c')$, $H''=(h_1'',\ldots,h_c'')$ and $H'+H''=(h_1'+h_1'',\ldots,h_c'+h_c'').$
\end{chunk}

\section{Chain level cohomology operators}
\label{s:extmodules}

Adopting the usual notation we also    fix \[
S=Q[\chi_1,\ldots,\chi_c\mid\gdeg{\chi_i}=\gdeg{f_i}^{-1},|\chi_i|=2,i=1,\ldots,c]
\] for the remainder of the article.  We will regard $S$ as a color DG $Q$-algebra with trivial differential.

\begin{chunk}\label{dualaction} We notice that $S$ can be realized as $\Hom_Q(\Ay{Q}{Y},Q)$ by \cref{dualpolyisdivpowalg} where $\chi_i$ is the $Q$-linear dual of $y_i$. We recall that the algebra structure on $S$ is induced by the coalgebra structure on $\Ay{Q}{Y}$.  That is, the product $\chi_i\chi_j$ is identified with the composition 
\[
\Ay{Q}{Y}\xra{\Delta}\Ay{Q}{Y}\otimes_Q \Ay{Q}{Y} \xra{\chi_i\otimes \chi_j}Q\otimes_Q Q\xra{\mu} Q
\] where the isomorphism is the multiplication map and $\Delta$ is defined as $\Phi$ in \Cref{equationforPhi}, namely
\[ \Delta(y^{(H)})=\sum_{H'+H''=H}\left(\prod_{i<j}\c(y_i,y_j)^{h'_jh''_i}\right)y^{(H')}\otimes y^{(H'')},
\]
where $H'=(h_1',\ldots,h_c')$ and $H''=(h_1'',\ldots,h_c'')$.
\end{chunk}

\begin{construction}
\label{c:1}
Let $X$ be a color DG $E^\e$-module.  Define $\E_X$ to be the DG $S$-module with underlying graded $Q$-module $S\otimes_Q X$ and differential
\[\del^{\E_X}=1\otimes \del^X+\sum_{i=1}^c \chi_i\otimes (\lambda_i-\lambda_i')\] where $\lambda_i$ and $\lambda_i'$ are left multiplication by  $1\otimes e_i$ and  $e_i\otimes 1$, respectively. 
Explicitly, as a graded $S$-module  $\E_X$ is  $$\E_X^\natural\cong \bigoplus_{j\in \mathbb{Z}} \mathsf{\Sigma}^{2j}(S\otimes_Q X_j) $$ 
and  on elements its differential is prescribed by 
\begin{align*}
    \del^{\E_X}(s\otimes x)&=s\otimes \del^X(x)+\sum_{i=1}^c \c(1\otimes e_i-e_i\otimes 1, s)\chi_i s\otimes (1\otimes e_i-e_i\otimes 1)x\\
   &=s\otimes \del^X(x)+\sum_{i=1}^c \c(f_i, s)\chi_i s\otimes (1\otimes e_i-e_i\otimes 1)x.\\
\end{align*}
\end{construction}

\begin{proposition}
$\E_X$ is a color DG  $S$-module.  
\end{proposition}
\begin{proof}
Note that $\del:=\del^{\E_X}$ is (color) $S$-linear. Indeed, continuing from the computation in Construction \ref{c:1} 
\begin{align*}
    \del(s\otimes x)  &=s\otimes \del^X(x)+\sum_{i=1}^c \c(f_i, s)\chi_i s\otimes (1\otimes e_i-e_i\otimes 1)x\\
    &=s\otimes \del^X(x)+\sum_{i=1}^c \c(f_i, s)\c(\chi_i,s)s\chi_i \otimes (1\otimes e_i-e_i\otimes 1)x\\
      &=s\otimes \del^X(x)+\sum_{i=1}^c s\chi_i \otimes (1\otimes e_i-e_i\otimes 1)x\\
      &=s\del(1\otimes x)
\end{align*}
where the third equality holds because $\c(\chi_i,-):=\c(f_i,-)^{-1}$.

It remains to show $\del^2=0$, but since $\del$ is (color) $S$-linear it suffices to check $\del^2(1\otimes x)=0$ for each $x$ in $X$. Consider 
\begin{align*}
    \del^2(1\otimes x)&= \del\left(1\otimes \del^X(x)+\sum_{i=1}^c\chi_i\otimes (1\otimes e_i-e_i\otimes 1)\cdot x \right)\\
    &=1\otimes \del^X\del^X(x)+ \sum_{i=1}^c\chi_i\otimes (1\otimes e_i-e_i\otimes 1)\cdot \del^X(x)\\
      & \ \ +\sum_{i=1}^c\chi_i\otimes \del^X((1\otimes e_i-e_i\otimes 1)\cdot x) \\
      &  \ \ +\sum_{i=1}^c\sum_{j=1}^c \c(f_i,\chi_j)\chi_i\chi_j\otimes  (1\otimes e_i-e_i\otimes 1) (1\otimes e_j-e_j\otimes 1)x.
\end{align*}
The first summand in the last expression above is evidently zero.  Also,  as  $X$ is a DG $E^\e$-module
\begin{align*}
\del^X((1\otimes e_i-e_i\otimes 1)x)&=\del^{E^\e}(1\otimes e_i-e_i\otimes 1)x-(1\otimes e_i-e_i\otimes 1)\del^X(x)\\
&=-(1\otimes e_i-e_i\otimes 1)\del^X(x),
\end{align*}
and so \begin{equation}
\del^2(1\otimes x)=\sum_{i=1}^c\sum_{j=1}^c \c(f_i,\chi_j) \chi_i\chi_j\otimes  (1\otimes e_i-e_i\otimes 1) (1\otimes e_j-e_j\otimes 1)x.\label{e:2}
\end{equation}
For $i=j$, 
\[ \c(f_i,\chi_j) \chi_i\chi_j\otimes  (1\otimes e_i-e_i\otimes 1) (1\otimes e_j-e_j\otimes 1)x=0.\] For $i\neq j$, we show that the two terms on the right-hand side of (\ref{e:2}) involving $i$ and $j$ cancel. Indeed, 
set $f:=f_i$, $f':=f_j$, $e:=1\otimes e_i-e_i\otimes 1$, $e':=1\otimes e_j-e_j\otimes 1$, $\chi=\chi_i$ and $\chi':=\chi_j$ for ease of notation. Since 
\[
\c(f,\chi')\chi \chi'=\c(f,f')^{-1}\chi\chi'=\c(f,f')^{-1}\c(f,f')\chi'\chi=\chi'\chi
\]
 the first equality below follows
\begin{align*}
    \c(f,\chi')\chi\chi'\otimes ee'+\c(f',\chi)\chi'\chi\otimes e'e
    &=\chi'\chi\otimes ee'+\c(f',\chi)\chi'\chi\otimes e'e\\
    &=\chi'\chi\otimes(ee'+ \c(f',f)^{-1} e'e)\\
    &=\chi'\chi\otimes (ee'-\c(f',f)^{-1}\c(f',f)ee')\\
    &=0.
\end{align*} 
Combining these calculations with (\ref{e:2}) it follows that $\del^2=0$, finishing the proof that $\E_X$ is a color DG $S$-module. 
\end{proof}

\begin{theorem}\label{p:1}
For any bounded above DG $E^\e$-module $X$, the map

\begin{center}
\begin{tikzpicture}[baseline=(current  bounding  box.center)]
 \matrix (m) [matrix of math nodes,row sep=0.5em,column sep=4em,minimum width=2em] {
 \eta_X:\E_X&\Hom_{E^\e}(\Ay{E^\e}{Y},X)\\
 \hphantom{\eta_X:}\quad\chi_i\otimes x&(y^{(H)}\mapsto\c(x,y^{(H)})\chi_i(y^{(H)})x)\\};
\path[->] (m-1-1) edge (m-1-2);
\path[|->] (m-2-1) edge (m-2-2);
\end{tikzpicture}
\end{center}
is an isomorphism color DG $Q$-modules. Moreover, $\eta_X$ satisfies
\begin{equation}
\label{eq:equivariant}
\eta_X(\chi_i\otimes x\cdot\chi_j)=\eta_X(\chi_i\otimes x)\cdot\eta_E(\chi_j\otimes1),
\end{equation}
where $x\in X$ (cf. \Cref{augmentedcup}).
That is, the isomorphism $\eta_X$ is equivariant with respect to the DG algebra map 
\[
S\to S\otimes_Q E\xra{\eta_E} \Hom_{E^\e}(\Ay{E^\e}{Y},E).
\] 
\end{theorem}
\begin{proof} Consider the following isomorphisms of graded $S$-modules
\begin{align}
    \Hom_{E^\e}(E^\e\langle Y\rangle, X)^\natural&\cong \Hom_Q(Q\langle Y\rangle, X)^\natural \label{e4}\\
    &\cong \Hom_Q(Q\langle Y\rangle, Q)\otimes_Q X^\natural \label{e5}\\
    &\cong S\otimes_Q X^\natural\label{e6}
\end{align} where (\ref{e4}) is adjunction, (\ref{e5}) follows as $Q\langle Y\rangle$ consists of degreewise finite rank free $Q$-modules and $X$ is bounded above, and (\ref{e6}) is discussed in  the appendix (cf. \Cref{dualpolyisdivpowalg}).

Now we check that the isomorphisms above send  $\del^{\Hom_{E^\e}(E^\e\langle Y\rangle, X)}$ to $\del^{\E_X}$. Indeed, 
\[
\del^{\Hom_{E^\e}(E^\e\langle Y\rangle, X)}=\Hom(E^\e\langle Y\rangle,\del^X)-\Hom(\del^{E^\e}\otimes1+\sum_{i=1}^c(\lambda_i'-\lambda_i)\otimes \chi_i, X)
\]
and so (\ref{e4}) maps this differential to 
\[
\Hom(Q\langle Y\rangle,\del^X)-\sum_{i=1}^c\c(f_i,\lambda_i'-\lambda_i)^{-1}\Hom( \chi_i, \lambda_i'-\lambda_i)
\] which, using that $\c(-,\lambda_i'-\lambda_i)=\c(-,f_i)$, is simply 
\[
\Hom(Q\langle Y\rangle,\del^X)-\sum_{i=1}^c\Hom( \chi_i, \lambda_i'-\lambda_i)=\Hom(Q\langle Y\rangle,\del^X)+\sum_{i=1}^c\Hom( \chi_i, \lambda_i-\lambda_i').
\] Next, (\ref{e5}) maps the differential to 
\[
\Hom_Q(Q\langle Y\rangle, Q)\otimes \del^X+\sum_{i=1}^c \Hom(\chi_i,Q)\otimes (\lambda_i-\lambda_i')
\]
and so (\ref{e6}) gives us exactly  $\del^{\E_X}$, as claimed.
Finally, the composition of the isomorphisms \Cref{e4}, \Cref{e5}, and \Cref{e6} is exactly $\eta_X.$

Now we prove \Cref{eq:equivariant} holds by showing both sides agree after evaluating  at $y^{(H)}$ and using the $E^\e$-linearity. Let $x\in X$ and $H\in \mathbb{N}^c$. 
First, observe that 
\[
\eta_X(\chi_i\otimes x\cdot\chi_j)(y^{(H)})=\c(x,\chi_jy^{(H)})\chi_i\chi_j(y^{(H)})x
\] where the product $\chi_i\chi_j$ is interpreted as in \Cref{dualaction}. 

Assuming that $i<j$ (the other cases are similar), 
 by the definition of $\chi_i\chi_j$ it follows that the previous display is zero unless $H$ has a 1 in position $i$ and $j$, and zero everywhere else. Hence, the display above is equal to\[
\c(x,y_i)x.
\]
Similarly, using \cref{equationforPhi}, we see that if $H$ has 1 in position $i$ and $j$ and zero everywhere else, then
\[
(\eta_X(\chi_i\otimes x)\cup\eta_E(\chi_j\otimes1))(y^{(H)})=\c(x,y_i)x,
\]
otherwise $(\eta_X(\chi_i\otimes x)\cup\eta_E(\chi_j\otimes1))(y^{(H)})=0$.
\end{proof}

\begin{remark}\label{importantremark}Let $X$ be a bounded above color DG $E^\e$-module.
Combining \Cref{hochschildsameascup} and  \Cref{p:1}, in homology $\eta_X$ induces an equivariant isomorphism
\[
\hh(\E_X)\xra{\hh(\eta_X)} \Ext_{E^\e}(E,X)
\] that respects the canonical projection  
$\pi: S\to S\otimes_Q R\cong \dHH{R}{Q};$
the left-hand side is regarded with the obvious $S$-action from \Cref{c:1} while the right-hand side has the composition action; see the beginning of  \cref{section:severalactions}.  

Viewing $\Ext_{E^\e}(E,X)$ as an $S$-module via restriction of scalars along $\pi$, the previous observation states that $\hh(\eta_X)$ is an isomorphism of graded $S$-modules.
Because of this we  identify these $S$-actions on $\Ext_{E^\e}(E,X). $
\end{remark}

\begin{chunk} \label{c:actions}
Let $F$ and $G$ be color DG $E$-modules. Then $\Hom_Q(F,G)$ naturally inherits the structure of a DG $E^\e$-module. Namely, given $\alpha\in \Hom_Q(F,G)$ (homogeneous with respect to all gradings) define 
\[1\otimes e_i\cdot \alpha:=(-1)^{|\alpha|}\c(e_i,\alpha) \alpha(e_i\cdot -) \ \text{ and } \
e_i\otimes 1\cdot \alpha:=e_i\cdot \alpha( -).\]  Hence, $\E_{\Hom_Q(F,G)}$ can be regarded as a DG $S$-module via  \Cref{c:1}. We will write $\E_{F,G}$ in lieu of $\E_{\Hom_Q(F,G)}$. Moreover, when $F^\natural$ and $G^\natural$ are free as graded  $Q$-modules, it follows that  $\E_{F,G}^\natural$ is a free  graded  $S$-module. 
\end{chunk}

\section{Cohomology operators on Ext modules}\label{s:FiniteGen}

\begin{definition}
\label{d:koszulres} Let $M$ be a color DG $E$-module. A surjective quasi-isomorphism of color DG $E$-modules $F\xra{\simeq} M$ is called a  \emph{Koszul resolution} of $M$ provided that $F$ is semifree over $Q$ via restriction of scalars along the structure map $Q\to E$. When $F$ is a finite DG $E$-module we say    the Koszul resolution is \emph{finite}. In this case, $F$ is  strongly perfect over $Q$. 
\end{definition}

\begin{proposition}\label{finitekosres}
 For each  finite color DG $E$-module $M$, there exists a finite Koszul resolution
$P\xra{\simeq} M$.
\end{proposition}
\begin{proof}
By \ref{prop:DGModRes}, there exists a semifree resolution $\epsilon: F\xra{\simeq }M$ of $M$ over $E$ where 
\[
F^\natural\cong \bigoplus_{j=i}^\infty \mathsf{\Sigma}^j (E^{\beta_j})^\natural
\] for some fixed $i\in \mathbb{Z}$ and nonnegative integers $\beta_j$. In particular, when $F$ is regarded as a complex of $Q$-modules via the structure map $Q\to E$, $F$ is a bounded below complex of finite rank free $Q$-modules. 
As $Q$ has finite global dimension and $M$ is  finite,  $\coker \del^F_{n+1}$ is free over $Q$ for each $n\gg 0$.

Now consider $F'$ defined as
\[
\ldots \to F_{n+2}\to F_{n+1}\to {\mathrm{im}\; \del^F_{n+1}}\to 0;
\] 
it is straightforward to see, simply by degree considerations, $F'$ is a DG $E$-submodule of  $F$. Furthermore, by possibly increasing $n$ one can assume that $F'$ is acyclic and concentrated in degrees strictly larger than $\max\{i:M_i\neq 0\}$.

Next, we take $P$ to be the quotient DG $E$-module $F/F'$
\[
0\to \coker \del^F_{n+1}\to F_{n-1}\xra{\del^F_{n-1}}\ldots \xra{\del^F_{i+1}} F_i\to 0.
\] As  $P$ is the quotient of $F$ by an acyclic DG $E$-submodule $F' $ such that  $\epsilon|_{F'}=0$, there is a canonically induced quasi-isomorphism $P\xra{\simeq} M$ of DG $E$-modules. Finally, we remark that by construction $P$ is strongly perfect over $Q$.
\end{proof}

\begin{chunk}\label{SmodstructureonExt}
Let $M$ and $N$ be finite color DG $R$-modules and fix  $F$  a Koszul resolution of $M$. The quasi-isomorphism  $E\xra{\simeq}R$ induces the first isomorphism of graded color $R$-modules below (see \cite[Proposition 6.7]{FHT} for a proof in the case that $E$ has trivial color), while the second isomorphism follows from \Cref{p1}(1)
\begin{equation}\label{gradediso}
\Ext_R(M,N)\xra{\cong} \Ext_E(M,N)\xra{\cong}\Ext_{E^\e}(E,\Hom_Q(F,N)).
\end{equation} The isomorphism in \Cref{gradediso} provides $\Ext_R(M,N)$ the structure of a $\dHH{R}{Q}$-module (and hence, via restriction of scalars, an $S$-module structure).
\end{chunk}

\begin{theorem}
\label{t:1} \label{t:fg}
Let $M$ and $N$ be  finite  color DG $R$-modules. There exist the following isomorphisms of graded $S$-modules
\[
    \Ext_R(M,N)\cong \hh(\E_{F,N})\cong \hh(\E_{F,G})
\]
where $F\xra{\simeq} M$ and $G\xra{\simeq} N$ are any bounded below Koszul resolutions. Moreover, $\Ext_R(M,N)$ is a finitely generated graded $S$-module.
\end{theorem} 
\begin{proof}
The isomorphisms follow from \Cref{p:1} and the discussion in  \Cref{SmodstructureonExt}. 

For the moreover statement,  let $F$ be a finite Koszul resolution of $M$, which exists by \Cref{finitekosres}.
 Since $F$ is strongly perfect over $Q$, \[
 \E_{F,N}^\natural=S\otimes_Q \Hom_Q(F,N)^\natural\] is a noetherian graded $S$-module. Any subquotient of a noetherian module is again noetherian and so the already established isomorphism of graded $S$-modules, above, implies that $\Ext_R(M,N)$ is a noetherian over $S$.
\end{proof}

\begin{proposition}\label{prop}
 The following hold for color DG $R$-modules $L,M,N$:
\begin{enumerate}
    \item The natural isomorphisms of graded color $Q$-modules 
     \begin{align*}
         \Ext_R(L,M \oplus N)&\cong  \Ext_R(L,M)\oplus \Ext_R(L, N)\\
          \Ext_R(L\oplus M,N)&\cong  \Ext_R(L,N)\oplus \Ext_R(M, N)\\
         \Ext_R(M,\mathsf{\Sigma}^n N)&\cong \mathsf{\Sigma}^n\Ext_R(M,N)\cong \Ext_R(\mathsf{\Sigma}^{-n}M,N)
   \end{align*}
    are isomorphisms of graded $S$-modules;
    \item For any exact sequence of color DG $R$-modules $0\to M^1\to M^2\to M^3\to 0$, the exact sequences of graded $Q$-modules   \begin{align*}
        \Ext_R(M^3,N)&\to \Ext_R(M^2,N)\to \Ext_R(M^1, N)\to \mathsf{\Sigma} \Ext_R(M^3,N)\\
          \Ext_R(L,M^1)&\to \Ext_R(L,M^2)\to \Ext_R(L,M^3)\to \mathsf{\Sigma} \Ext_R(L,M^1)
    \end{align*}  are exact sequences of graded $S$-modules. 
    \end{enumerate}
\end{proposition}
\begin{proof}
All of these follow directly from \Cref{importantremark}; we will  prove the first isomorphism in (1) while the rest are left to the reader. 
Fix a Koszul resolution $F\xra{\simeq}L$,  the claim follows from the commutativity of the following diagram where all arrows are isomorphisms of graded color $Q$-modules
\begin{center}
\begin{tikzpicture}[baseline=(current  bounding  box.center)]
 \matrix (m) [matrix of math nodes,row sep=3em,column sep=2em,minimum width=2em] {
 \Ext_R(L,M\oplus N)&\Ext_R(L,M)\oplus\Ext_R(L,N)\\
 \Ext_E(L,M\oplus N)&\Ext_E(L,M)\oplus\Ext_E(L,N)\\
 \Ext_{E^\e}(E,\Hom_Q(F,M\oplus N))&\Ext_{E^e}(E,\Hom_Q(F,N))\oplus\Ext_{E^e}(E,\Hom_Q(F,N)).\\};
\path[->] (m-1-1) edge (m-1-2);
\path[->] (m-2-1) edge (m-2-2);
\path[->] (m-3-1) edge (m-3-2);
\path[->] (m-1-1) edge (m-2-1);
\path[->] (m-2-1) edge (m-3-1);
\path[->] (m-1-2) edge (m-2-2);
\path[->] (m-2-2) edge (m-3-2);
\end{tikzpicture}
\end{center}
Both vertical maps at the top of the diagram are induced by the quasi-isomorphism $E\rightarrow R$, see \Cref{SmodstructureonExt}. The second set of vertical maps are those from \Cref{p1}(1). The bottom horizontal map is clearly $\dHH{R}{Q}$-linear, and 
so the desired result follows from \Cref{importantremark}.
\end{proof}

\section{Support varieties}\label{section:supportvarieties}

Let $A$ be a commutative noetherian  graded  $\kk$-algebra that is concentrated in  even nonnegative  cohomological degrees. We let $\D^f(A)$ denote the bounded derived category of  finite DG $A$-modules which is obtained in the standard way of formally inverting quasi-isomorphisms between DG $A$-modules. Explicitly, the objects of $\D^f(A)$ are DG $A$-modules whose homology is a finitely generated graded $A$-module.

\begin{chunk}
Let  $\Proj A$ denote the topological space consisting of homogeneous prime ideals not containing $A^{>0}$ equipped with the Zariski topology. The closed subsets of $\Proj A$  are of the form 
\[
\cV(g_1,\ldots, g_t)=\{\mathsf{p}\in \Proj A: g_i\in \mathsf{p} \text{ for all }i\}
\]
for some homogeneous elements $g_1,\ldots, g_t\in A$. For a graded $A$-module $Y$ and $\mathsf{p}\in \Proj A$, we let $Y_\mathsf{p}$ denote the homogeneous localization of $Y$ at $\mathsf{p}$. Also, for $\mathsf{p}\in \Proj A$ define $k(\mathsf{p})$ to be $A_{\mathsf{p}}/\mathsf{p} A_{\mathsf{p}}.$
\end{chunk}

\begin{chunk}\label{c:2}
 For a DG $A$-module $X$, its (small) support is  \[\supp_AX=\{\mathsf{p}\in \Proj A: X\dt_A k(\mathsf{p})\not\simeq 0\}.\] By \cite[Theorem 2.4]{CI}, if $X\in \D^f(A)$ then 
\[ \supp_AX=\{\mathsf{p}\in \Proj A: \hh(X)_\mathsf{p}\neq 0\}.\]
Furthermore, as $\hh(X)$ is a finitely generated graded $A$-module, the support of $X$ over $A$ is exactly $\cV(g_1,\ldots, g_t)$ where $g_1,\ldots, g_t$ is some list of homogeneous generators for $\ann_A \hh(X)$. Thus,  $\supp_AX$ is a closed  subset of $\Proj A$ whenever $X\in \D^f(A)$.
\end{chunk}
\begin{chunk}\label{c:tensor} We recall the following well known property of  cohomological support (see, for example, \cite[2.1.5]{P:CS}). For DG $A$-modules $X$ and $X'$, 
\[\supp_A(X\dt_A X')=\supp_AX\cap \supp_A X'.
\]
In particular, when $X$ is semiprojective over $A$, $X\otimes_A X'\simeq X\dt_A X'$ and so  \[\supp_AX\cap \supp_A X'=\supp_A(X\otimes_A X').\]
\end{chunk}

\begin{chunk}\label{a:1}\label{c:4}
For the rest of the section we add to our fixed notation from \Cref{s:semifreediagonal} the assumption that the group of colors of $Q$ is finite. We point out that this hypothesis is equivalent to saying that the skew commuting parameters of $Q$ are roots of unity.  In particular, 
there exists  $t>0$ such that the graded subalgebra 
\[
A:=Q'[\chi_1^t,\ldots, \chi_c^t]
\] of $S$ is \emph{commutative} where $Q'$ is the subalgebra on the generators for $Q$ raised to the $t^{\text{th}}$ power. 
 Moreover, it is clear that $A\subseteq S$  is a module finite extension and $A$ has finite global dimension. 
\end{chunk}

 With \cref{a:1} in place, we  have a way to study graded Ext-modules over $R$ as modules over a commutative polynomial ring in variables of cohomological degree $2t$ with coefficients in $Q'$. This allows us to introduce a theory of support varieties analogous to the ones over commutative complete intersections as well as more general (commutative) settings (cf. \cite{VPD,AB,AI,BW,Jor,P:CS}). 
 
\begin{definition} If $M$ and $N$ are color DG $R$-modules, we define the \emph{support variety of $(M,N)$} to be 
\[\V_R(M,N)=\supp_A(\Ext_R(M,N)).\]
\end{definition}

\begin{chunk}\label{c:3}
Assume that $M$ and $N$ are  finite color DG $R$-modules. In this case, $\Ext_R(M,N)$ is  a finitely generated graded  $A$-module. Indeed,  by Theorem \ref{t:fg}, $\Ext_R(M,N)$ is  a finitely generated graded color $S$-module. Since $A\to S$ is module finite, the claim holds. That is, $\E_{F,N}$ is an object of $\D^f(A)$ where $F\xra{\simeq} M$  is a finite Koszul resolution of $M$. Hence, 
\begin{align*}
    \V_R(M,N)&=\{\mathsf{p}\in \Proj A: \Ext_R(M,N)_\mathsf{p}\neq 0\}\\
    &=\{\mathsf{p}\in \Proj A: \E_{F,N}\dt_A k(\mathsf{p})\not\simeq 0\}\\
     &=\{\mathsf{p}\in \Proj A: \E_{F,G}\dt_A k(\mathsf{p})\not\simeq 0\}
\end{align*}where  $G\xra{\simeq} N$ is a finite Koszul resolution of $N$; the first and second equalities hold by \cref{c:2}, and the third equality is justified by Theorem \ref{t:1}. 
\end{chunk}

From \cref{c:2} and \cref{c:3}, the support variety of a pair of  finite color DG $R$-modules is in fact a closed subset of $\Proj A.$ This is recorded in the first proposition below. 
\begin{proposition}If $M$ and $N$ are   finite color DG $R$-modules, then $\V_R(M,N)$ is a closed subset of $\Proj A.$ \qedhere
\end{proposition}

\begin{proposition}
\label{p:basicprops}
The following hold for DG $R$-modules $L,M,N$:
\begin{enumerate}
    \item $\V_R(L\oplus M,N)= \V_R(L,N)\cup \V_R(M, N).$ 
     \item    $\V_R(L,M\oplus N)= \V_R(L,M)\cup \V_R(L,N).$
     \item For any $n\in \mathbb{Z}$, 
     $\V_R(M,\mathsf{\Sigma}^n N)=\V_R(M,N)=\V_R(\mathsf{\Sigma}^n M,N).$
    \item For  $0\to M^1\to M^2\to M^3\to 0$ an  exact sequence of color DG $R$-modules  \begin{align*}
        \V_R(M^h,N)&\subseteq \V_R(M^i,N)\cup \V_R(M^j, N)\\
         \V_R(N,M^h)&\subseteq \V_R(N,M^i)\cup \V_R(N,M^j)
    \end{align*} whenever $\{h,i,j\}=\{1,2,3\}.$
        \item If $\Ext_R^n(M,N)=0$ for all $n\gg 0$, then $\V_R(M,N)=\emptyset$; the converse holds when both $M$ and $N$ are  finite DG $R$-modules. 
\end{enumerate}
\end{proposition}
\begin{proof}
The first four statements are clear from \ref{prop} and standard facts for homogeneous supports of graded modules over \emph{commutative} noetherian graded rings (as $A$ is); for these facts see, for example,   \cite[2.2]{AI}. For (5), the forward implication is elementary. For the converse, the assumption on $M$ and $N$ imply that  $\Ext_R(M,N)$ is a finitely generated graded $A$-module, see \Cref{c:3}, and so we can  apply \cite[2.2(5)]{AI} directly. 
\end{proof}

\begin{example}
We show that $\V_R(\kk,\kk)=\Proj (A\otimes_{Q'} \kk)$. We first notice that by \cite[Theorem 10.7]{ColorDGA} it follows that $\Ext_R(\kk,\kk)$ is isomorphic as a $S$-module, and hence as a $A$-module, to $S\otimes_Q\bigwedge\!\!{}^\mathbf{q}(\kk e_1\oplus\cdots\oplus \kk e_n)$, see \cite[Definition 10.3]{ColorDGA} for the definition of skew exterior algebra.  This tensor product is isomorphic to a direct sum of copies of $A\otimes_{Q^\prime}\kk $, therefore $\V_R(\kk,\kk)=\Proj (A\otimes_{Q'} \kk)$.
\end{example}

\begin{example}
Let $Q=\mathbb{C}_{\mathbbm{i}}[x,y], R=\frac{Q}{(x^2,y^2)}$ and let $M=\frac{R}{(x)}$, in this example we are going to calculate $\V_R(M,\CC)$. A $Q$-resolution of $M$ is given by the skew Koszul complex of the sequence $x,y^2$
\[
F:0\longrightarrow Q\xra{\begin{pmatrix}y^2\\x\end{pmatrix}} Q^2\xra{\begin{pmatrix}x&y^2\end{pmatrix}} Q\longrightarrow0.
\]
Let $E$ be the skew Koszul complex over the ring $Q$ of the sequence $x^2,y^2$, and let $e_1,e_2$ be the basis elements that differentiate to $x^2,y^2$ respectively. The complex $F$ admits a structure of color DG $E$-module by defining the action in the following way:
\[
e_1:F_0\xra{\begin{pmatrix}x\\0\end{pmatrix}} F_1,\quad e_1:F_1\xra{\begin{pmatrix}0&x\end{pmatrix}}F_2,
\]
\[
e_2:F_0\xra{\begin{pmatrix}0\\1\end{pmatrix}}F_1,\quad e_2:F_1\xra{\begin{pmatrix}1&0\end{pmatrix}}F_2.
\]
We calculate $\Ext_R(M,\CC)$ by calculating the homology of the complex $\E_{F,\CC}$.  
As $F$ is minimal,
the differential given in \ref{c:1} reduces to left multiplication by  $\chi_2\otimes\lambda_2$, using the actions given in \ref{c:actions}. 

For the rest of the computation let $(-)^*:=\Hom_Q(-,\mathbb{C})$.
It is directly  observed the DG $S$-module $\E_{F,\CC}$ can be regarded as the skew Koszul complex over $S\otimes_Q\CC$ of the sequence $(\chi_2,0)$:
\[
0\rightarrow \mathsf{\Sigma}^{-4}S\otimes_QF_2^*\xrightarrow{\begin{pmatrix}\chi_2\\0\end{pmatrix}} \mathsf{\Sigma}^{-2}S\otimes_QF_1^*\xrightarrow{\begin{pmatrix}0&\chi_2\end{pmatrix}} S\otimes_QF_0^*\rightarrow 0.
\]
 Therefore, $\Ext_R(M,\CC)$ is free over $\frac{\CC[\chi_1,\chi_2]}{(\chi_2)}$. Furthermore,  $A=\CC[\chi_1^2,\chi_2^2]$ and hence,  \[\V_R(M,\CC)=\supp_A\left(\frac{\CC[\chi_1,\chi_2]}{(\chi_2)}\right)=\V_R(\chi_2^2).\]
\end{example}

\begin{example}\label{exampleperfect}A   finite color DG $R$-module $M$  is perfect over $R$  if and only if $\V(M,\kk)=\emptyset$ (cf. Proposition \ref{p:basicprops}(5)).  In fact, for  a perfect color DG $R$-module $M$,  $\V_R(M,N)=\emptyset$ for all  finite color DG $R$-modules $N$. 
In particular, it follows from this remark and Proposition  \ref{p:basicprops}(4) that if $M$ is a finitely generated color $R$-module then 
\[
\V_R(M,N)=\V_R(\Omega_R^i(M),N)
\] for any  finite color DG $R$-module $N$ and any $i\geq 0$ where  $\Omega_R^i(M)$ denotes the $i^{\text{th}}$ syzygy module of $M$ over $R$. 
\end{example}


The following proof is adapated from \cite[Theorem 4.3.1]{P:CS} by  working over the smaller \emph{commutative} ring $A$, rather than $S$. We sketch the argument for the convenience of the reader. 
 \begin{theorem}\label{t:symmetry}
For  finite color DG $R$-modules  $M,M',N,N'$, 
\[\V_R(M,N)\cap \V_R(M',N')=\V_R(M,N')\cap \V_R(M',N).\]
\end{theorem}
\begin{proof}
By \cref{c:3}, we can replace $M,M',N,N'$ with their finite Koszul resolutions and so in the sequel we assume these are all strongly perfect over $Q$. In particular, $\E_{X,Y}^\natural$ is a finite rank free graded $S$-module for $X=M,M'$ and $Y=N,N'$. As $A\to S$ is a finite free extension of graded $A$-modules, $\E_{X,Y}^\natural$ is a finite rank free graded $A$-module for $X=M,M'$ and $Y=N,N'$. By applying \cite[Proposition 1.2.8]{P:CS}, $\E_{X,Y}$ is a semiprojective DG $A$-module; here the fact that $A$ has finite global dimension is essential (cf. \ref{c:4}). Thus, 
\begin{equation}
    \E_{X,Y}\dt_A\E_{X',Y'}=\E_{X,Y}\otimes_A\E_{X',Y'}\label{iso}
\end{equation}whenever $\{X,X'\}=\{M,M'\}$ and $\{Y,Y'\}=\{N,N'\}$.
Also, we have the following isomorphisms of graded $A$-modules
\begin{align*}
\E_{X,Y}\otimes_A\E_{X',Y'}&= (S\otimes_Q \Hom_Q(X,Y))\otimes_A (S\otimes_Q \Hom_Q(X',Y'))\\
&\cong (S\otimes_A S) \otimes_Q \Hom_Q(X,Y)\otimes_Q \Hom_Q(X',Y')\\
&\cong (S\otimes_A S) \otimes_Q \Hom_Q(X,Y')\otimes_Q \Hom_Q(X',Y)\\
&\cong (S\otimes_Q \Hom_Q(X,Y'))\otimes_A (S\otimes_Q \Hom_Q(X',Y))\\
&\cong \E_{X,Y'}\otimes_A\E_{X',Y}
\end{align*} 
where the third isomorphism is induced  from the  natural  evaluation map  \[\Hom_Q(P,Q)\otimes_Q V\xra{\cong} \Hom_Q(P,V)\]  being an isomorphism for  a strongly perfect $Q$-complex $P$  and for any complex $V$.
 Tracing the differentials through the isomorphisms above verifies that this in fact establishes an isomorphism of DG $A$-modules  \[\E_{X,Y}\otimes_A\E_{X',Y'}\cong 
\E_{X,Y'}\otimes_A\E_{X',Y}.\]
Combining this with (\ref{iso}) establishes the following isomorphim of DG $A$-modules
\[
 \E_{M,N}\dt_A\E_{M',N'}\cong\E_{M,N'}\dt_A\E_{M',N}. 
\] Thus, by \ref{c:tensor} 
\[
\supp_A\E_{M,N}\cap \supp_A\E_{M',N'}=\supp_A\E_{M,N'}\cap \supp_A\E_{M',N}
\]
and so \ref{c:3} yields the desired result.
\end{proof}

We obtain the following corollary from the symmetry of supports satisfied in Theorem \ref{t:symmetry}. Namely, Corollary \ref{c:symmetry} is a  consequence of Theorem \ref{t:symmetry}; since the argument is the same as in  \cite[4.3.1]{P:CS} we omit its proof here. 
\begin{corollary}\label{c:symmetry}
For  any pair of   finite color DG $R$-modules $M,N$, $\V_R(M,N)=\V_R(N,M).$ Moreover, the following closed subsets of $\Proj A$ coincide \begin{enumerate}

    \item $\supp_{A\otimes_{Q'}\kk}\left(\Ext_R(M,N)\otimes_{Q'}\kk\right)$;
    \item $\supp_{A\otimes_{Q'}\kk}\left(\Ext_R(N,M)\otimes_{Q'}\kk\right)$;
    \item $\V_R(M,N)\cap \V_R(\kk,\kk)$
       \item $\V_R(M,\kk)\cap \V_R(\kk,N)$;
        \item $\V_R(M,\kk)\cap \V_R(N,\kk)$;
           \item $\V_R(\kk,M)\cap \V_R(\kk,N)$.
\end{enumerate} In particular, $
\V_R(M,M)\cap \V_R(\kk,\kk)=\V_R(M,\kk)=\V_R(\kk,M).$
\end{corollary}

\section{Vanishing of Ext modules}\label{s:ARC}

We continue with the usual hypothesis that $R$ is a skew complete intersection as in \Cref{s:semifreediagonal}  and the hypothesis used in the last section that  the group of colors of $Q$ is finite. We apply the facts from the previous section to obtain the following results over such skew complete intersections.

\begin{proposition}
\label{finiteprojdim}
Let $R$ be a skew complete intersection with a finite group of colors. For a finite color DG $R$-module  $M$, the following are equivalent: 
\begin{enumerate}
    \item $M$ is perfect over $R$; 
    \item $\Ext_R^{\gg 0}(M,M)=0$. 
\end{enumerate} In particular, if $M$ is a finitely generated color $R$-module, then $M$ has finite projective dimension over $R$ if and only if  $\Ext_R^{\gg 0}(M,M)=0$. 
\end{proposition}
\begin{proof}
The  implication ``(1) implies (2)" is trivial,  so we assume $\Ext_R^{\gg 0}(M,M)=0.$ By Proposition \ref{p:basicprops}(5), it follows that $\V_R(M,M)=\emptyset. $ It follows from \cref{c:symmetry} that 
\[
\emptyset=\V_R(M,M)=\V_R(M,M)\cap\V_R(\kk,\kk)=\V_R(M,\kk).
\]
Now the desired result is obtained from  Example \ref{exampleperfect}.
\end{proof}

\begin{proposition}
Let $R$ be a skew hypersurface with finite group of colors. If $M$ and $N$ are  finite color DG $R$-modules such that  $\Ext_R^{\gg 0}(M,N)=0$, then $M$ or $N$ is a perfect DG $R$-module. In particular, if $M$ and $N$ are finitely generated color $R$-modules such that $\Ext_R^{\gg 0}(M,N)=0$, then $\pd_R M<\infty$ or $\pd_RN<\infty.$
\end{proposition}
\begin{proof}
By \cref{c:symmetry} it follows that
\begin{equation}
\emptyset=\V_R(M,N)=\V_R(M,N)\cap\V_R(\kk,\kk)=\V_R(M,\kk)\cap\V_R(N,\kk).\label{emptyint} 
\end{equation}
The assumption that $R$ is a hypersurface implies  $\Proj \kk[\chi_1^t]=\{(0)\}$ and so    $\V_R(M,\kk)$ and $\V_R(N,\kk)$ are  naturally identified with subsets of $\{(0)\}.$
 Therefore, from (\ref{emptyint})  it follows that  one of $\V_R(M,\kk)$ or $\V_R(N,\kk)$  must be empty. Hence,  $M$ or $N$ is perfect over $R$ (cf. Example \ref{exampleperfect}). \end{proof}
\begin{theorem}[Generalized Auslander-Reiten Conjecture]\label{theoremARC}
Let $R$ be a skew complete intersection with finite group of colors and let  $M$ be a finitely generated color $R$-module. If 
$
\Ext_R^i(M,M\oplus R)=0$ for all $i>r$, then $\pd_R M\leq r$.
\end{theorem}

\begin{proof}
It follows from \Cref{finiteprojdim} that $M$ has finite projective dimension. The theorem now follows from the following 
\begin{equation}\label{equationGARC}
\pd_RM=\mathrm{sup}\{i\mid\Ext_R^i(M,R)\neq0\},
\end{equation}
as the right hand side is clearly at most $r$ by assumption. Hence, we prove that \Cref{equationGARC} holds. 

First, by the graded  version of Nakayama's Lemma 
\begin{align*}
\pd_RM&=\mathrm{sup}\{i\mid\Ext_R^i(M,\kk)\neq0\}.
\end{align*} Since $M$ has finite projective dimension \Cref{equationGARC} holds by the fact above and by using the exact sequence $0\rightarrow R_+\rightarrow R\rightarrow\kk\rightarrow0$, where $R_+$ is the ideal generated by the elements of $R$ of positive internal degree.
\end{proof}

\begin{remark}
It is proved in \cite{Schulz} that the ring $\kk_\mathfrak{q}[x_1,x_2]/(x_1^2,x_2^2)$ does not satisfy the Generalized Auslander-Reiten Conjecture whenever $q_{1,2}$ is not a root of unity. We point out that the module considered in \cite{Schulz} is $(x_1+x_2)$, which is \emph{not} a color module. It is unknown whether  a skew complete intersection with infinite group of colors satisfies the Generalized Auslander-Reiten Conjecture on \emph{color} modules.
\end{remark}

\section{Symmetry in complexity}\label{s:cx}

We continue with the usual assumption that $R=Q/(f_1,\ldots,f_c)$ is a skew complete intersection.

\begin{chunk}
Let $\{b_i\}$ be a sequence of nonnegative integers. Recall that the complexity of $\{b_i\}$, denoted by $\cx \{b_i\}$, is the least integer $d$ such that there exists $a>0$ satisfying 
$b_i\leq ai^{d-1}$ for all $i\gg 0. $

Let  $M$ and $N$ be  finite color DG $R$-modules.
The \emph{complexity} of the pair $(M,N)$ is defined as 
$$\cx_R(M,N) = \cx \{ \dim_\kk(\Ext_R^i(M,N)\otimes_R \kk) \}.$$
\end{chunk}
In \cite[Corollary 10.10]{ColorDGA} the first two authors show  $\cx_R(\kk,\kk)=c$.
 In fact, a stronger statement about the \emph{Poincar\'{e} series} of $\kk$ is determined. To elaborate on this we introduce the following notation. 
 \begin{chunk}
 Let $M$ and $N$ be finite color $R$-modules. The \emph{Poincar\'{e} series of} $(M,N)$ is 
 \[
 {\rm P}^R_{M,N}(t):=\sum_{i\geq 0} \dim_\kk(\Ext_R^i(M,N)\otimes_R \kk)t^i.
 \]Also, we define the  \emph{Poincar\'{e} series of $M$} to be 
$  {\rm P}^R_{M}(t):= {\rm P}^R_{M,\kk}(t).$ Finally, for a (cohomologically) graded  $\kk$-module $\mathcal{M}$ we let its \emph{Hilbert series} be
\[
{\rm H}_{\mathcal{M}}(t)=\sum_{i\in \mathbb{Z}}\dim_\kk \mathcal{M}^it^i. 
\]
 \end{chunk} 
The stronger statement from \cite[Corollary 10.8]{ColorDGA}, mentioned above, says 
 \[
 {\rm P}^R_\kk(t)=\dfrac{(1+t)^n}{(1-t^2)^c}
 \]  and so there exists a polynomial $p$ of degree $c-1$ such that 
 \[
 p(i)=\dim_\kk\Ext_R^i(\kk,\kk)
 \] for all $i \gg 0. $
 The following theorem  generalizes the facts above to arbitrary pairs of modules over a skew complete intersection. We emphasize that \cref{finitecomplexity}, and its corollaries, do not require the assumption from the previous section that the parameters $q_{i,j}$ should be roots of unity.

\begin{theorem}\label{finitecomplexity}
Let $R$ be a skew complete intersection of codimension $c$ and $M$ and $N$ be finite color $R$-modules. If $\Ext_R(M,N)\otimes_R\kk\neq0$ , then the formal power series
\[
(1-t^2)^{\cx_R(M,N)}{\rm P}
^R_{M,N}(t)
\]
is a polynomial with integer coefficients that has no root at $t=1$.
In particular, $\cx_R(M,N)\leq c$.
\end{theorem}
\begin{proof}
By Theorem \ref{t:fg}, $\Ext_R(M,N)$ is a finitely generated color module over $S = Q[\chi_1,\dots,\chi_c\mid \gdeg{\chi_i}=\gdeg{f_i}^{-1},|\chi_i|=2]$,
a color commutative polynomial ring.  
Hence, $\mathcal{M}:=\Ext_R(M,N)\otimes_Q \kk$ is a finitely generated color module over $\mathcal{S}:=S\otimes_Q \kk.$ As  each $\chi_i$ has cohomological degree 2, $\mathcal{M}$ decomposes as a direct sum of finitely generated color $\mathcal{S}$-modules
\[
\mathcal{M}=\mathcal{M}^{\text{even}}\oplus \mathcal{M}^{\text{odd}}. 
\] Also, 
we have the equalities 
\[
{\rm P}^R_{M,N}(t)={\rm H}_{\mathcal{M}}(t)={\rm H}_{\mathcal{M}^{\text{even}}}(t)+{\rm H}_{\mathcal{M}^{\text{odd}}}(t), 
\]and 
\[
\cx_R(M,N)=\max\left(\cx\{\dim_\kk \mathcal{M}^{2i}\}, \cx\{\dim_\kk \mathcal{M}^{2i+1}\}\right).
\]
By the previous two displays, it suffices to prove the desired result when $\mathcal{M}$ is concentrated solely even degrees or solely in odd degrees; so we assume, without loss of generality, $\mathcal{M}$ is concentrated in even degrees. Therefore, for the rest of the proof we  can regrade $\mathcal{S}$ by  assigning each $\chi_i$ cohomological degree 1, and $\mathcal{M}$ will be a finitely generated $\mathcal{S}$-module.

Next, as $\mathcal{S}$ has finite global dimension it follows that $\mathcal{M}$ admits a bounded  resolution by finite rank free color $\mathcal{S}$-modules. Now using that the Hilbert series is additive along exact sequences and 
${\rm H}_{\mathcal{S}}(t)=(1-t)^{-c}$
 it follows that 
\[
{\rm H}_{\mathcal{M}}(t)=\frac{q(t)}{(1-t)^c}
\] for some polynomial $q(t)$ with integer coefficients. 
By canceling the common factors of $(1-t)$ we can write
\[
{\rm H}_{\mathcal{M}}(t)=\frac{p(t)}{(1-t)^{c'}}
\] for some polynomial with integer coefficients $p(t)$, where $p(1)\neq 0$ and $c'\leq c$. Now a direct calculation shows that $c'$ is exactly $\cx\{\dim_k \mathcal{M}^i\},$ as needed.
\end{proof}


\begin{corollary}
Every finitely generated color $R$-module over a skew complete intersection has rational Poincar\'{e} series. 
\end{corollary}

\begin{corollary}\label{corsemifreres}
Let $M$ be a finite color $R$-module with a minimal free resolution $F\xra{\simeq} M$ over $R$. Then 
$\cx\{{\rm rank}_R F_i\} \leq c.$
\end{corollary}
\begin{proof}
This follows immediately from \Cref{finitecomplexity} with $N=\kk.$
\end{proof}
\begin{remark}
More can be said about specific resolutions over $R$. Namely, let $M$ be a finite color DG $R$-module and fix a Koszul resolution $F\xra{\simeq} M$. The same argument as in  \cite[Theorem 2.4]{AB:CD2} shows that the resolution of the diagonal in  \Cref{prop:Eres} determines an $R$-semifree resolution of $M$, depending on the choice of $F$. In particular, $M$ will admit a semifree $R$-resolution whose underlying graded $R$-module is   
\[R\otimes_Q \Hom_Q(S,Q)\otimes_Q F.\]  From this one obtains a second proof of \Cref{corsemifreres} since $F$ can be taken to be a strongly perfect $Q$-complex (cf. \Cref{SmodstructureonExt}). 
\end{remark}

For the remainder of the section we assume $Q$ is a
skew polynomial ring with each $q_{i,j}$ a root of unity. 
\begin{theorem}\label{symmetriccomplexity}
Let $R$ be a skew complete insersection with finite group of colors.
If $M$ and $N$ are  finite color DG $R$-modules,  then
$\cx_R(M,N) = \cx_R(N,M)$.
\end{theorem}

\begin{proof}Adopting the notation set from \Cref{a:1}, 
 $\Ext_R(M,N)$
is a finitely generated module over the $S$-subalgebra $A$. The upshot is that we may instead compute complexity using $A$ rather than $S$.  Also, since $Q'\to Q$ is a module finite extension we have 
\begin{equation}\label{equationeasiercx}
\cx_R(M,N)  =  \cx \{ \dim_\kk(\Ext_R^i(M,N)\otimes_{Q'} \kk) \}.\end{equation}
Therefore, 
\begin{align*}
\cx_R(M,N) & =   \cx \{ \dim_\kk(\Ext_R^i(M,N)\otimes_{Q'} \kk) \}\\
        & =   \dim_A \Ext_R(M,N) \otimes_{Q'} \kk \\
        &= \dim \supp_A \Ext_R(M,N) \otimes_{Q'} \kk \\
        &= \dim \supp_A \Ext_R(N,M) \otimes_{Q'} \kk \\
        & =   \dim_A   \Ext_R(N,M) \otimes_{Q'} \kk \\
        &=\cx \{ \dim_\kk(\Ext_R^i(N,M)\otimes_{Q'} \kk) \}\\
        & =  \cx_R(N,M);
\end{align*}
where the first and last equalities are justified by \Cref{equationeasiercx}, the fourth equality is from \Cref{c:symmetry}, and the rest are standard since we are working over the graded commutative ring $A$ (see, for example, \cite[Section 4.1]{BH}).
\end{proof}

\begin{corollary}
Let $R$ be a skew complete intersection with finite group of colors. Let $M$ and $N$ be  finite color DG $R$-modules, then
\[\Ext_R^{\gg 0}(M,N)=0 \iff \Ext_R^{\gg 0}(N,M)=0.\]
\end{corollary}

\section*{Acknowledgements}
We  thank Benjamin Briggs for references regarding braided Hochschild cohomology and Jason Gaddis for sharing his notes on the Generalized Auslander-Reiten Conjecture.

\appendix
\section{Skew divided powers algebra} \label{a:2}

In this appendix, we show that the dual of a color polynomial ring under the convolution
product is isomorphic (as an algebra) to a color divided powers algebra and vice versa.
For background regarding skew divided powers algebras, see \cite[Section 6]{ColorDGA}.

Let $Q$ be a color commutative $\kk$-algebra.  That is, we assume that $Q$
admits a $G$-grading where $G$ is an abelian group, and that $ab = \c(\sigma,\tau)ba$ for all
$G$-homogeneous elements $a \in Q_\sigma$ and $b \in Q_\tau$, where $\c : G \times G \to \kk^*$ 
is an alternating bicharacter of $G$.  This bicharacter $\c$ is fixed throughout.
When necessary, we denote the $G$-degree of a homogeneous element by $\calG(x)$,
and we will abuse notation and write $\c(a,b)$ for $\c(\calG(a),\calG(b))$.

Next, we let $A=Q[x_1,\ldots,x_n\mid \gdeg{x_i}=\sigma_i, |x_i|=d_i,i=1,\ldots,n]$, with $\sigma_i\in G$ and $d_i$ positive integers. We denote a monomial in $A$ by $\bsx^{\bsalpha}$ where
$\bsalpha = (\alpha_1,\dots,\alpha_n) \in \bbN^n$ is an exponent vector.


Let $\Delta : A \to A \otimes_Q A$ be the coproduct defined by setting
$\Delta(x_i) = x_i \otimes 1 + 1 \otimes x_i$ and extending by products and linearity.
For a general element of $A$, we use Sweedler's notation and let
$\Delta(a) = \sum a_{(1)} \otimes a_{(2)}$.
Note that this coproduct is bihomogeneous for the canonical bigrading on $A \otimes_Q A$.
Using the definition of $\Delta$, a straightforward calculation shows that
\begin{equation}\label{coprodMon}
\Delta(\bsx^\bsalpha) = \sum_{\bsbeta + \bsgamma = \bsalpha} C(\bsx^\bsbeta,\bsx^\bsgamma)^{-1}
  \binom{\bsalpha}{\bsbeta} \bsx^\bsbeta \otimes \bsx^\bsgamma,
\end{equation}
where $\binom{\bsalpha}{\bsbeta} = \binom{\alpha_1}{\beta_1}\cdots\binom{\alpha_n}{\beta_n}$, 
and $C(\bsx^\bsbeta,\bsx^\bsgamma)$ is the element of $\kk^*$ that satisfies
$\bsx^\bsbeta\bsx^\bsgamma = C(\bsx^\bsbeta,\bsx^\bsgamma)\bsx^{\bsbeta + \bsgamma}$.
Note that $C(-,-)$ is a bicharacter defined on the monoid of monomials
(but is not alternating), and satisfies
\begin{equation} \label{chiAndC}
\c(\bsx^\bsbeta,\bsx^\bsgamma) =
C(\bsx^\bsbeta,\bsx^\bsgamma)C(\bsx^\bsgamma,\bsx^\bsbeta)^{-1}.
\end{equation}
See \cite{FMM} for more details regarding the $C(-,-)$ pairing.

Let $A^*=\mathrm{Hom}_Q(A,Q)$,
note that $A^*$ is a free left and right $Q$-module, and is spanned by the
dual basis of monomials $(\bsx^{\bsalpha})^*$;
we denote the element $(\bsx^\bsalpha)^*$ by $\bsxi^\bsalpha$.
Further, if the bidegree of $\bsx^{\bsalpha}$ is $(\sigma,d)$, then the bidegree
of $(\bsx^\alpha)^*$ is $(\sigma^{-1},-d)$.

We may now define a product on $A^*$ using the coproduct on $A$.  Indeed,
for any $\varphi,\psi \in A^*$ and $a \in A$ bihomogeneous:
\begin{equation} \label{convolProd}
(\varphi\psi)(a) = \sum \c(\psi,a_{(1)})\varphi(a_{(1)})\psi(a_{(2)}).
\end{equation}

The reason for the appearance of the $\c$ factor is due to its presence in the
canonical isomorphism $V^* \otimes_Q W^* \cong (V \otimes_Q W)^*$ in the category
of color $Q$-modules.

Before continuing, we prove a lemma regarding this product,
as it demonstrates some of the ideas needed to justify claims which follow.
\begin{lemma}
Let $\bsbeta$ and $\bsgamma$ be exponent vectors, and let $\bsxi^\bsbeta$, $\bsxi^\bsgamma$
be the duals of $\bsx^\bsbeta$ and $\bsx^\bsgamma$, respectively.  Then one has
$$
\bsxi^\bsbeta\bsxi^\bsgamma = C(\bsx^\bsgamma,\bsx^\bsbeta)^{-1}\binom{\bsbeta + \bsgamma}{\bsbeta}
                              \bsxi^{\bsbeta+\bsgamma}.
$$
\end{lemma}
\begin{proof}
The binomial coefficient present on the right hand side comes from the binomial coefficient
present in the coproduct formula \eqref{coprodMon} above.  By the
coproduct formula and the definition of the convolution product \eqref{convolProd},
the additional unit factor
is equal to $\c(\bsxi^\bsgamma,\bsx^\beta)C(\bsx^\bsbeta,\bsx^\bsgamma)^{-1}$.  Since
\[
\c(\bsxi^\bsgamma,\bsx^\bsbeta)C(\bsx^\bsbeta,\bsx^\bsgamma)^{-1}
 = \c(\bsx^\bsbeta,\bsx^\bsgamma)C(\bsx^\bsbeta,\bsx^\bsgamma)^{-1}
 = C(\bsx^\bsgamma,\bsx^\bsbeta)^{-1},
\]
the result follows.
\end{proof}

The algebra $A^*$ also carries a system of skew divided powers.
These are defined by using divided power binomial expansion and the following
definition of divided power of a monomial:
$$(\bsxi^\bsalpha)^{(k)} = C(\bsx^\bsalpha,\bsx^\bsalpha)^{-\binom{k}{2}}
\combBracket{\bsalpha}{k}\bsxi^{k\bsalpha}$$
where for $h$ and $k$ integers, one has $\combBracket{h}{k} = \dfrac{(hk)!}{k!(h!)^k}$, and for $\bsalpha$ an exponent vector, one has
$\combBracket{\bsalpha}{k} = \combBracket{\alpha_1}{k}\cdots\combBracket{\alpha_n}{k}$.

The proof that $A^*$ satisfies the axioms of a skew divided powers algebra follows
from careful use of the fact that $C(-,-)$ and $\c(-,-)$ are bicharacters,
together with identities involving binomial coefficients and the bracket notation 
introduced above.  We provide here a proof of the following equality as an example.
\begin{proposition}
Let $x$ and $y$ be elements of $A^*$ that are homogeneous with respect to all gradings.  Then one has
$$(xy)^{(k)} = \c(y,x)^{\binom{k}{2}}x^ky^{(k)}.$$
\end{proposition}
\begin{proof}
Using divided powers binomial expansion, it is enough to prove this formula for
$x = \bsxi^\bsbeta$ and $y = \bsxi^\bsgamma$.  In this case, it is clear that
both sides of the claimed equality evaluate to a scalar multiple of 
$\bsxi^{k(\bsbeta+\bsgamma)}$.  The scalar on the left hand side is made up of
a constant involving $C(-,-)$, and combinatorial constants:
$$
C(\bsx^\bsgamma,\bsx^\bsbeta)^{-k}
C(\bsx^{\bsbeta+\bsgamma},\bsx^{\bsbeta+\bsgamma})^{-\binom{k}{2}}
\binom{\bsbeta+\bsgamma}{\bsbeta}^k
\combBracket{\bsbeta+\bsgamma}{k}.
$$
By induction, it follows that 
$$
(\bsxi^\bsbeta)^k = C(\bsx^\bsbeta,\bsx^\bsbeta)^{-\binom{k}{2}}
                    \binom{2\bsbeta}{\bsbeta}\cdots\binom{k\bsbeta}{\bsbeta}
                    \bsxi^{\bsbeta}.
$$
Therefore,  the scalar on the right hand side is therefore the product of
$$
\c(\bsx^\bsgamma,\bsx^\beta)^{\binom{k}{2}}
C(\bsx^\bsbeta,\bsx^\bsbeta)^{-\binom{k}{2}}
C(\bsx^\bsgamma,\bsx^\bsbeta)^{-\binom{k}{2}}
C(\bsx^{k\bsgamma},\bsx^{k\bsbeta})^{-1}
$$
and the combinatorial constant
$$
\binom{2\bsbeta}{\bsbeta}\cdots\binom{k\bsbeta}{\bsbeta}
\combBracket{\bsgamma}{k}
\binom{k(\bsbeta+\bsgamma)}{k\bsbeta}.
$$
It is a straightforward matter to check that the combinatorial constants on either
side of the equality agree.  To check that that the constants involving $\c(-,-)$ and $C(-,-)$
agree, one uses that $C(-,-)$ is a bicharacter \eqref{chiAndC}.
\end{proof}

We may use the product on $A$ to define a coproduct on $A^*$ as well, by declaring
$\Delta(\xi_i) = \xi_i \otimes 1 + 1 \otimes \xi_i$ and extending $\Delta$ to all of $A^*$ as in
\eqref{coprodMon}.  Using this coproduct, we may in turn consider the algebra $A^{**}$ as before.
However, one has the following result, whose proof follows along the same lines as in the
case of a commutative ring $Q$:

\begin{proposition}\label{dualpolyisdivpowalg}
The algebras $A$ and $A^{**}$ are isomorphic as color commutative graded $Q$-algebras.
In particular, the graded $Q$-dual of a skew divided powers algebra over $Q$ is isomorphic to
a color polynomial extension of $Q$.
\end{proposition}

\section{An adjunction isomorphism}

The  goal of this appendix is to prove \Cref{p1}. In the commutative case, this follows immediately from  \cite[(8.7)]{Mac}. 

Let $A$ be a graded color commutative DG $Q$-algebra where $Q$ is a color commutative connected graded algebra over a fixed base field $\kk$; set  $A^\e:=A\otimes_Q A^{\op}$ to be the enveloping DG algebra of $A$ over $Q$ and, as usual,  $A$ is regarded as a color DG $A^\e$-algebra via the multiplication map $\mu: A^\e\to A$ given by $$a\otimes b\mapsto ab.$$ 

Fix a color DG A-module $M$, then we  define a pair of functors (depending on $M$) \[\h: \Mod(A)\to \Mod(A^\e)\text{ and }\te: \Mod(A^\e)\to \Mod(A)\] given by 
$\h:=\Hom_Q(M,-)$ and $\te:=-\otimes_A M,$ respectively. 
The $A^\e$-module structure on $\h(N)=\Hom_Q(M,N)$ is given by 
\[
a\otimes b \cdot f :=(-1)^{|f||b|}\c(b,f)af(b-)\] while the $A$-action on $\te(N)$ is obvious one on the left of $N$ in $N\otimes_ AM$. 

\begin{chunk}
Let $N$ be a color DG $A$-module. We let $\ev_N:\Hom_Q(M,N)\otimes_A M\to N$ be the evaluation map. Namely, $$\ev_N(g\otimes m):=g(m).$$ It is straightforward to check  that $\ev_N$ is a morphism of color DG $A$-modules. 
\end{chunk}

\begin{chunk}
Fix   a color DG $A^\e$-module $X$ and  a color DG $A$-module  $N$. Let $f: X\otimes_A M\to N$ be a morphism of color DG $A$-modules and consider $\tilde{f}: X\to \Hom_Q(M,N)$ given by $$x\mapsto f(x\otimes -).$$ Observe that 
\begin{align*}
    \tilde{f}(a\otimes b\cdot x)&=f((-1)^{|b||x|}\c(b,x)axb\otimes-) \\
    &=(-1)^{|b||x|}\c(b,x)f(axb\otimes-)\\
    &=(-1)^{|a||f|+|b||x|}\c(b,x)\c(f,a)af(xb\otimes-)\\
    &=(-1)^{|a||f|+|b||x|}\c(b,x)\c(f,a)af(x\otimes b -)\\
    &=(-1)^{|a||f|+|b||x|+|b||f|+|b||x|}\c(b,x)\c(f,a)\c(\tilde{f}(x),b)a\otimes b \cdot \tilde{f}(x)\\
    &=(-1)^{(|a|+|b|)|\tilde{f}|} \c(\tilde{f},a\otimes b) a\otimes b\cdot \tilde{f}(x),
\end{align*} that is, $\tilde{f}$ is left color $A^\e$-linear. 
\end{chunk}

\begin{proposition}\label{p1}
With the notation above, $\h$ is a right adjoint to $\te.$ In particular, for each DG $A^\e$-module $X$ and DG $A$-module $N$ the following maps are inverse isomorphisms that are natural in $X$, $M$ and $N$:
\begin{enumerate}
    \item $\Phi: \Hom_A(\te(X),N)\to \Hom_{A^\e}(X,\h(N))$ given by 
    $$f\mapsto \tilde{f}$$
    \item $\Psi: \Hom_{A^\e}(X,\h(N))\to \Hom_A(\te(X),N) $ given by $$g\mapsto \ev_N \circ (g\otimes \id^M).$$
\end{enumerate}
\end{proposition} 
\begin{proof}
We directly check  the maps defined in (1) and (2) are mutually inverse to one another. To see this consider
\begin{align*}
    \Psi\Phi(f)(x\otimes m)&=\ev_N\circ (\tilde{f}\otimes \id^M)(x\otimes m)\\
    &=\ev_N(f(x\otimes-)\otimes m)\\
    &=f(x\otimes m),
\end{align*} 
and 
\begin{align*}
    \Phi\Psi(g)(x) &=\Phi\left(\ev_N\circ (g\otimes \id^M)\right)(x)\\
    &=\ev_N(g(x)\otimes -)\\
    &=g(x).
\end{align*} Therefore $\Phi\Psi=\id$ and  $\Psi\Phi=\id$, justifying  the proposition.
\end{proof}

\bibliographystyle{amsplain}
\bibliography{biblio}
\end{document}